\documentclass{amsart}

\usepackage{mathtools}
\usepackage{amssymb,amsmath}
\usepackage{graphicx}
\usepackage{pinlabel}
\usepackage{pdfpages}
\usepackage{lineno}
\usepackage{xspace}
\usepackage{comment}
\usepackage{accents}
\usepackage{subcaption}
\usepackage[colorlinks=true,allcolors=blue]{hyperref}
\usepackage{xcolor}%

\usepackage{fancyhdr}
\usepackage{lipsum}

\newcommand{\bG}{\mathbb{G}}
\newcommand{\E}{\mathcal{E}}
\newcommand{\V}{\mathcal{V}}
\newcommand{\Q}{\mathcal{Q}}
\newcommand{\gec}[1]{\mathcal{#1}}

\newcommand{\calM}{\mathcal{M}}
\newcommand{\dr}[1]{\underline{#1}}
\newcommand{\cog}[1]{\textbf{\textit #1}}
\newcommand{\ass}[1]{#1}

\def\wil{{}^\odot}

\def\ouco{(\mskip-3mu(}
\def\cuco{)\mskip-3mu)}
\def\ea{E_{\textup{a}}}
\def\ev{E_{\textup{v}}}
\def\ef{E_{\textup{f}}}

\DeclareMathOperator{\Match}{Match}

\newcommand{\hsplice}{\rotatebox[origin=c]{90}{$\Join$}}
\newcommand{\vsplice}{\mathbin{\rotatebox[origin=c]{90}{$=$}}}

\newcommand{\spa}[1]{\mathbf{#1}}

\newcommand{\ba}{\mathbin{\backslash}}
\newcommand{\ext}{\mathbin{\dagger}}

\DeclareMathOperator{\seg}{seg}
\DeclareMathOperator{\sat}{sat}
\DeclareMathOperator{\cro}{cr}

\DeclareMathOperator{\ksplit}{split}
\DeclareMathOperator{\ksplice}{splice}
\DeclareMathOperator{\ktotal}{total}
\newcommand{\tcon}{\mathbin{\rightthreetimes}}

\newtheorem{theorem}{Theorem}[section]
\newtheorem{corollary}[theorem]{Corollary}

\newtheorem{lemma}[theorem]{Lemma}
\newtheorem{proposition}[theorem]{Proposition}

\theoremstyle{definition}

\newtheorem{definition}[theorem]{Definition}
\newtheorem{example}[theorem]{Example}

\theoremstyle{remark}
\newtheorem{remark}[theorem]{Remark}

\numberwithin{equation}{section}

\begin{document}

\title{Polynomial invariants of cyclically ordered graphs}

\author{Paul Bratch}
\address{Department of Mathematics, Royal Holloway, University of London, Egham, TW20~0EX, United Kingdom}
\email{paul.bratch.2022@live.rhul.ac.uk}

\author{M. N. Ellingham}
\address{Department of Mathematics, 1326 Stevenson Center, Vanderbilt University, Nashville, Tennessee 37240, U.S.A.}
\email{mark.ellingham@vanderbilt.edu}

\author{Joanna A. Ellis-Monaghan}
\address{Korteweg de Vries Institute for Mathematics, University of Amsterdam, 1098 XH Amsterdam, The Netherlands}
\email{j.a.ellismonaghan@uva.nl}

\author{Iain Moffatt}
\address{Department of Mathematics, Royal Holloway, University of London, Egham, TW20~0EX, United Kingdom}
\email{iain.moffatt@rhul.ac.uk}

\author{Wout Moltmaker}
\address{Korteweg de Vries Institute for Mathematics, University of Amsterdam, 1098 XH Amsterdam, The Netherlands}
\email{w.c.moltmaker@uva.nl}


\date{\today}


\begin{abstract}
 Cyclically ordered graphs, or \emph{cogs}, sit between abstract graphs and  cellularly embedded graphs.    They arise naturally in topological graph theory, knot theory, and mathematical biology. 
We develop  a formal theory of cogs and establish a number of  invariants of cogs. In particular we detail several ways to present cogs and detail how these descriptions can be used to construct cog invariants by adapting  the matching, transition and Yamada polynomials.
\end{abstract}


\maketitle


\section{Introduction}
We develop the theory of cyclically ordered graphs, or \emph{cogs}. 
Cogs arise in several distinct areas of mathematics including topological graph theory, knot theory, and mathematical biology. In this paper we consolidate the various approaches to cogs and present a common theory for them. We then define and analyse a range of polynomial-valued cog invariants that will serve as tools for further understanding the properties of cogs. 

A cog is an abstract graph with an undirected cyclic order at each vertex.
 Cogs sit between abstract graphs and graphs cellularly embedded in surfaces, where the latter can be described by graphs with not only (directed) cyclic orders at the vertices but also signs on the edges (see Subsection~\ref{ss:rs} for details).  For example, while there is exactly one abstract graph with two edges and one vertex (up to graph isomorphism), and six cellular embeddings of this graph (up to equivalence), there are exactly two cogs derived from this graph (up to cog isomorphism, defined later).

Our interest in cogs stems in part from topological graph theory and Chmutov's partial duals~\cite{zbMATH05569114}. It is a classical result that two embedded graphs are geometric duals (also known as Euler--Poincar\'e duals or surface duals)  if and only if their medial graphs are equal as embedded graphs. 
If instead we ask when two embedded graphs are partial duals then this is characterised by their embedded medial graphs having isomorphic cogs (see \cite{MR2869185} for details). Cogs also appear in the context  of \emph{rigid vertex graphs} in knot theory and in DNA applications, for example in \cite{B+15,B+13,Kau89}.

Given the success of graph polynomials as tools for understanding graphs and cellularly embedded graphs, it is natural to seek polynomials that serve analogously for cogs.  There are already many polynomials encoding information about abstract graphs, and also quite a few for graphs embedded in surfaces.  However, we know of no prior polynomial that captures the intermediate structure, between abstract and embedded graphs, recorded by cogs.  We begin filling this gap.

The observation that cogs interpolate between abstract graphs and cellularly embedded graphs motivates our strategy for constructing polynomial invariants of cogs.  We look either for ways to add structure to known abstract graph invariants so that they record information about the cog, or for ways to remove structure from known topological graph invariants so that they become cog invariants. Here we illustrate our strategies by constructing cog invariants arising from three very different invariants: the matching polynomial in Section~\ref{sec:matching_polynomial}, the transition polynomial in Section~\ref{sec:trans2}, and the Yamada polynomial in Section~\ref{ss.yam}. Each section demonstrates a different general approach to constructing cog invariants. We conclude with a discussion of some further directions of study in Section~\ref{sec:further}.

\section{Background}\label{sec:back}

In this section we review some common concepts from topological graph theory, give a formal definition of cogs, and contrast and compare cogs and cellularly embedded graphs. We conclude by defining {graph-encoded cogs}, or {gecs}, which give a description of cogs analogous to Lin's description of embedded graphs as graph-encoded maps, or gems~\cite{zbMATH03728291}. The different descriptions of cogs provide us with different approaches to constructing cog invariants in the subsequent sections. 
The key topological graph theoretic equivalences we review in this section may be summarised as follows.
\begin{itemize}
\item Embedded graphs  $\leftrightarrow$ signed rotation systems up to vertex flips $\leftrightarrow$ ribbon graphs.
\item Oriented embedded graphs $\leftrightarrow$ rotation systems.
\item Orientable embedded graphs $\leftrightarrow$ rotation systems up to component reversal.
\end{itemize}
We show the following are equivalent. 
\begin{itemize}
\item Cogs.
\item  Rotation systems up to vertex reversals.
\item Signed rotation systems up to vertex flips and  arbitrarily changing edge signatures.
\item Ribbon graphs or embedded graphs up to taking partial Petrie duals.
\item Graph encoded cogs (gecs).
\end{itemize}

\medskip

We allow graphs to have multiple edges and loops, and think of each edge as consisting of two half-edges.
Therefore, our formal definition of a graph $G$ has a set of vertices $V(G)$, a set of half-edges that are paired into edges, with edge set $E(G)$, and an incidence function that maps each half-edge to a vertex.
An isomorphism between graphs has a bijection between the vertex sets and a bijection between the sets of half-edges that preserve pairing into edges and incidence.  We consider graphs up to isomorphism.
In certain derived structures (gecs, generalised gecs, pointed-gecs and medial graphs) we also allow graphs to have \emph{free loops}. A free loop is an edge that is not incident to any vertex.  Free loops are often drawn as circles.
Each free loop is considered to be a cycle of length $0$, and counts as a component of $G$.

We write $k(G)$ for the number of components of a graph $G$. 
We let $G \ba e$  denote the graph obtained by deleting the edge $e$, and $G / e$  denote the graph obtained by contracting the edge $e$.
For a vertex $v$ we let $G \ba v$  denote the result of deleting $v$  from $G$, that is removing $v$ along with all its incident edges. 
For a set  of edges $A$ we write $G \ba A$, respectively $G / A$, for the result of deleting, respectively contracting, all of the edges in $A$. Similarly we write $G \ba U$ for the result of deleting all the vertices in a set $U$ of vertices.
Finally, we let  $G\ext e$ denote  the graph resulting from  \emph{extracting} an edge $e=(u,v)$ from $G$, which is defined to be the graph $G \ba \{u, v\}$, where possibly $u=v$.
Note that if $e$ and $f$ have no end points in common, then $(G\ext e) \ext f = (G\ext f) \ext e = G \ba U$, where $U$ is the set of the endvertices of $e$ and $f$.  However, if $e$ and $f$ share any endvertices, then neither $(G\ext e) \ext f$ nor $(G\ext f) \ext e$ is defined.

\medskip

Surfaces in this paper are compact but not necessarily connected $2$-manifolds.  If a surface can be given a globally consistent ``clockwise'' direction, it is \emph{orientable}, and the surface together with the clockwise direction is an \emph{oriented surface}.  Otherwise the surface is \emph{nonorientable}.
Every connected orientable surface corresponds to two oriented surfaces, since there are two choices of clockwise direction.

We assume that the reader is familiar with embeddings of graphs in surfaces and their various descriptions, and give only brief reminders of the constructions that are necessary here.  All embeddings in this paper are \emph{cellular}, with every face homeomorphic to an open disc, unless otherwise indicated.  We consider embeddings up homeomorphisms of the surface that induce graph isomorphisms.  When considering \emph{oriented embeddings}, in oriented surfaces, the homeomorphisms should be orientation-preserving.

\medskip

We also assume familiarity with \emph{cyclic orders} and with \emph{undirected cyclic orders} (sometimes called \emph{separation relations}).  These can be defined axiomatically as ternary and quaternary relations, respectively (see for example \cite{Hun35}), but we just note that a cyclic ordering of a finite set $S$ can be represented by a list $(x_1, x_2, \dots, x_s)$ of the elements of $S$ up to cyclic shifts, and an undirected cyclic ordering can be represented by a list $\ouco x_1, x_2 ,\dots, x_s \cuco$ of the elements of $S$ up to cyclic shifts and  reversal. If $|S| \le 3$ then $S$ has a unique undirected cyclic ordering.

%

\subsection{Cogs in the context of rotation systems}\label{ss:rs}
As is well-known (see for example \cite[Section 3.2]{GT87} or \cite[Section 3.3]{MT01}), cellular embeddings of graphs in surfaces can be described combinatorially using signed rotation systems.
A \emph{rotation system} is a graph together with a cyclic order of the half-edges incident to each vertex.  An \emph{edge signature} is a function mapping each edge of a graph to a sign, $+$ or $-$.  A \emph{signed rotation system} consists of a rotation system and edge signature for the same graph.
Two rotation systems, edge signatures, or signed rotation systems are \emph{isomorphic} if there is an isomorphism between their underlying graphs that preserves the cyclic order at each vertex and/or the edge signature, as appropriate.

Given a vertex $v$, a rotation system can be acted on by \emph{vertex reversal} at $v$, which reverses the cyclic order of half-edges incident with $v$.  An edge signature can be acted on by \emph{vertex switching} at $v$, which changes the sign of every non-loop edge incident with $v$.
A signed rotation system can be acted on by the combination of a vertex reversal at $v$ and a vertex switching at $v$, which we call a \emph{vertex flip} at $v$.
By applying sequences of these operations, two rotation systems (or signed rotation systems) can be \emph{reversal-equivalent}, two edge signatures (or signed rotation systems) can be \emph{switching-equivalent}, and two signed rotation systems can be \emph{flip-equivalent}.

Given a cellular graph embedding, we construct a signed rotation system by assigning a local clockwise orientation at each vertex, taking the half-edges incident with each vertex in their local clockwise order to get a rotation system, and assigning an edge a $+$ sign if translating the local clockwise orientations at one endvertex along the edge gives the local clockwise orientation at the other endvertex, and a $-$ sign if not.
Conversely, a signed rotation system describes how to glue discs onto a graph to obtain an embedding in a surface (see \cite[Section 3.2]{GT87} or \cite[Section 3.3]{MT01}). 
Cellular embeddings of $G$ are in 1-1 correspondence with flip-equivalence classes of signed rotation systems (each flip corresponds to changing the local clockwise direction at some vertex).

Oriented embeddings can also be represented using signed rotation systems, but to reflect the oriented nature of the surface we  use the specified clockwise direction at each vertex.  Because the direction is globally consistent, all edges receive a $+$ sign, so the rotation system alone describes the embedding.  
This gives a 1-1 correspondence between rotation systems for a graph $G$ and oriented embeddings of $G$.  Therefore, vertex reversal induces a well-defined operation on oriented embeddings.

Orientable embeddings of a graph $G$ can also be represented using signed rotation systems with all signs $+$, and hence just by rotation systems, by choosing some globally consistent clockwise direction.  There are two choices for this direction on each component of the surface, and making different choices corresponds to taking \emph{component reversals} in the rotation system, where we reverse every vertex in a given component of $G$.  Orientable embeddings therefore correspond to \emph{component-reversal-equivalence} classes of rotation systems.
Because individual vertex reversals commute with component reversals, vertex reversal is a well-defined operation on component-reversal-equivalence classes of rotation systems, and hence induces one on orientable embeddings.
An orientable embedding can usually also be described by signed rotation systems with some $-$ signs, but every such signed rotation system is flip-equivalent to one with all signs $+$.

\medskip

Cogs discard some of the information in a rotation system or signed rotation system. 
\begin{definition}
A \emph{cyclically ordered graph} or \emph{cog} is a graph with an \emph{undirected} cyclic ordering of the half-edges incident with each vertex.
Two cogs are \emph{isomorphic} if there is an isomorphism between their graphs that preserves the undirected cyclic order at each vertex.
\end{definition}

Each rotation system (or signed rotation system) has an \emph{underlying cog}, where we consider the cyclic order at each vertex up to reversal to obtain an undirected cyclic order.
This gives a 1-1 correspondence between cogs and reversal-equivalence classes of rotation systems, or of \emph{oriented} embeddings.
There is also a 1-1 correspondence between cogs and reversal-equivalence classes of \emph{orientable} embeddings.
Moreover, flip-equivalent signed rotation systems have reversal-equivalent rotation systems.  Thus, each flip-equivalence class of signed rotation systems, which is to say each graph embedding, has a unique underlying cog.
Furthermore, two signed rotation systems will have the same underlying cog if and only if their rotation systems are reversal-equivalent (in particular, the signatures are ignored).

For example, consider a theta graph, which consists of two vertices with three parallel edges between them.
There are two ways to assign a cyclic order at each vertex, giving four rotation systems and hence four oriented embeddings (two mirror images in the sphere or plane, and two mirror images in the torus).
But there is only one undirected cyclic order at each vertex, hence only one cog.  There are eight edge signatures, and hence $32$ signed rotation systems, which under flip-equivalence give eight different embeddings.  Up to isomorphism there are only two rotation systems and one cog, and up to homeomorphisms that induce graph isomorphisms there are four embeddings (one planar, one toroidal, one in the projective plane and one in the Klein bottle).

The above example of theta graphs is an instance of a general observation. Since there is a unique undirected cyclic ordering for any set of at most three elements, a graph of maximum degree at most $3$ has a unique cog. Thus, such cogs are equivalent to graphs. Of course, this is not true for cogs in general. For example, Table~\ref{tab:cogs} lists the connected cogs on three edges up to isomorphism. The convention in the table is the two half-edges forming an edge are given the same number, and the cycles specify the cyclic orders at the vertices. 

\begin{table}
\centering
\begin{tabular}{|c|c|}\hline
No. vertices & Connected cogs with 3 edges  \\ \hline
1 & (121323), (121233), (122133), (112233), (123123) \\ \hline
2 & (11223)(3), (12213)(3), (12123)(3), (123)(123), \\ 
 & (12)(1233), (12)(1323), (112)(233) \\ \hline
3 & (12)(123)(3), (1)(23)(123), (112)(3)(23), (123)(1)(23), (1)(1232)(3), \\ 
 & (1)(1223)(3), (1)(12)(233), (12)(133)(2), (12)(13)(23) \\ \hline
4 & (1)(2)(123)(3), (1)(23)(12)(3), (1)(2)(23)(13), (1)(12)(3)(23)
\\ \hline
\end{tabular}

\caption{Connected cogs on 3 edges.}
\label{tab:cogs}
\end{table}

\medskip

While a graph embedding has a unique underlying cog that does not depend on the signed rotation system used to represent it, two distinct graph embeddings may have the same cog.  To characterise when this happens we consider changing the signs of a set of edges in a signed rotation system.  Any such operation commutes with vertex flips, so it is a well-defined operation on flip-equivalence classes of signed rotation systems, and hence on graph embeddings.  These operations have been used since at least the late 1970s (see for example \cite{Sta78}) and were formalized in the context of ribbon graphs as \emph{partial Petrie duals} (\emph{partial Petrials} for short) \cite{MR2869185} (see Subsection \ref{ss:cogrg}).
Specifically, if $A$ is a set of edges in a graph $G$, then two flip-equivalence classes of signed rotation systems for $G$ (and hence the corresponding embeddings of $G$) are \emph{partial Petrials with respect to $A$} if a representative of the second class can be obtained from a representative of the first class by changing the signs of the edges in $A$.  This means that every representative of the second class can be obtained from every representative of the first class by first changing the signs of the edges in $A$ and then performing vertex flips.

\begin{proposition}\label{prop:ppcog}
Two graph embeddings have the same underlying cog if and only if they are partial Petrials.
Hence there is a 1-1 correspondence between partial Petrie duality classes of graph embeddings and cogs.
\end{proposition}

\begin{proof}
Suppose two embeddings have the same cog.  We may represent each by a signed rotation system, and then apply a sequence of vertex flips to one of them (which does not change the embedding) so as to make their rotation systems the same.  Then one is obtained from the other merely by changing some edge signs, so they are partial Petrials.

Conversely, suppose two embeddings are partial Petrials.  Then the signed rotation system for the second can be obtained from the signed rotation system for the first by (1) changing the signs of some edges and then (2) applying vertex flips.  Step (1) does not change the rotation system and step (2) applies vertex reversals to the rotation system, so the rotation systems are reversal-equivalent and hence have the same cog.
\end{proof}

We shall call the correspondence between the set of partial Petrial duals of a graph embedding and a cog established in Proposition~\ref{prop:ppcog} the \emph{natural correspondence} between these two objects.

\begin{remark}
The term cog was first used in~\cite[Section 3.2]{MR3086663}.  The definition of cogs used in this paper differs slightly from that give in ~\cite{MR3086663}.
There cogs were defined as graphs with a cyclic order of half-edges at each vertex (i.e., as rotation systems), but were considered up to reversal-equivalence. Thus, the cogs in this paper correspond to equivalence classes of cogs in \cite{MR3086663}. The definition of cogs used here avoids using two different names for the same type of object. As~\cite{MR3086663} was only interested in equivalence classes of cogs, this change in definitions is largely immaterial. 
\end{remark}

\subsection{Cogs in the context of ribbon graphs}\label{ss:cogrg} 
We review  a few relevant properties of ribbon graphs.   Additional background can be found in~\cite{MR3086663,zbMATH07553843}.

 A {\em ribbon graph} $\bG=\left(V,E\right)$ is a surface with boundary, represented as the union of two sets of closed discs --- a set $V$ of {\em vertices} and a set $E$ of {\em edges} --- such that: (1) the vertices and edges intersect in disjoint line segments; (2) each such line segment lies on the boundary of precisely one vertex and precisely one edge; and (3) every edge contains exactly two such line segments. 
As every ribbon graph $\bG$ can be regarded as a surface with boundary, it has some number of boundary components, which we shall denote by $b(\bG)$.

Two ribbon graphs are \emph{equivalent} if there is a homeomorphism that sends vertices to vertices, and edges to edges. We consider ribbon graphs up to this equivalence. 

Ribbon graphs are easily seen to be equivalent to cellularly embedded graphs.  Sewing discs into the boundary components of a ribbon graph, and then taking a deformation retraction reducing the ribbon vertices to points and the ribbon edges to curves in the resulting surface gives the corresponding cellularly embedded graph.  In the other direction, starting with a graph cellularly embedded in a surface and then cutting out a small neighbourhood (and its boundary) of the graph results in the corresponding ribbon graph.
Since ribbon graphs are equivalent to cellularly embedded graphs, a ribbon graph also corresponds to a flip-equivalence class of signed rotation systems.

To form the \emph{underlying graph} $G$ of a ribbon graph $\bG$ we cut each edge disc $e$ into two discs, each containing one of the line segments in which $e$ intersects a vertex disc.  These become the half-edges of $G$, and pairings of half-edges and incidence of half-edges with vertices are defined by non-empty intersection. Note that $\bG$ is a ribbon graph representation of a cellular embedding of its underlying graph.

Following~\cite{MR2869185}, the \emph{partial Petrie dual} or \emph{partial Petrial} of a ribbon graph $\bG$ with respect to an edge $e$ is the ribbon graph $\bG^{\tau(e)}$ obtained from $\bG$ as follows.  We begin by  detaching an end of $e$ from its incident vertex $v$ and specifying arcs $[a,b]$ on $v$ and $[a',b']$ on $e$ such that $\bG$ is recovered by identifying $[a,b]$ with $[a',b']$.  We then  reattach the detached end of $e$  by identifying the arcs antipodally (so that $[a,b]$ is identified with $[b',a']$). The result of this process is indicated in Figure~\ref{fig.pp}.
Partial Petrials with respect to single edges commute, so partial Petrials with respect to sets of edges are well defined.
The \emph{Petrie dual} (cf.~\cite{MR547621}) is the ribbon graph obtained by forming the partial Petrial with respect to the set of all edges.
Note that the ribbon graph $(\bG^{\tau(e)})^{\tau(e)}$ is equivalent to $\bG$.

If we are constructing $\bG^{\tau(e)}$ and we have a local clockwise direction at each endvertex of $e$, this process reverses the way in which those directions translate along the edge $e$, and so this corresponds to changing the sign of an edge in a signed rotation system representing $\bG$, agreeing with our definition of partial Petrial via signed rotation systems.
Proposition \ref{prop:ppcog} therefore applies with `ribbon graph' replacing `graph embedding'.
\begin{proposition}\label{prop:ppcog2}
Two ribbon graphs have the same underlying cog if and only if they are partial Petrials.
Hence there is a 1-1 correspondence between partial Petrie duality classes of graph embeddings and cogs.
\end{proposition}

\begin{figure}[t]
\centering
\begin{tabular}{ccc}
\labellist
\small\hair 2pt
\pinlabel $a=a'$ at 84 7
\pinlabel $b=b'$ at 84 84
\endlabellist
\includegraphics[height=20mm]{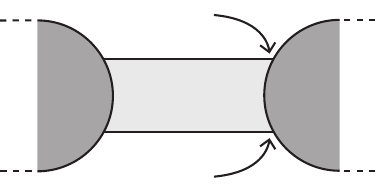}  &
\raisebox{9mm}{\includegraphics[width=15mm]{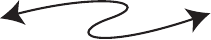}} &
\labellist
\small\hair 2pt
\pinlabel $a=b'$ at 84 7
\pinlabel $b=a'$ at 84 84
\endlabellist
\includegraphics[height=20mm]{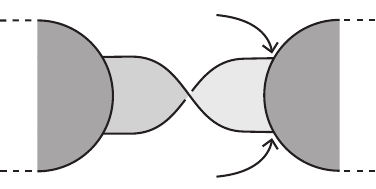}  
\end{tabular}
\caption{Forming a partial Petrial $\bG^{\tau(e)}$ at an edge $e$ of a ribbon graph $\bG$ .}
\label{fig.pp}
\end{figure}

\medskip

\begin{remark}
An intuitive model of cog classes arises from Proposition~\ref{prop:ppcog2}.  
Just as ribbon graphs may be thought of as ``graphs with discs for vertices and ribbons for edges'',  
cogs may be thought of as ``graphs with discs for vertices and strings for edges''.
\end{remark}

We will also work with medial graphs of embedded graphs.  Recall that if a graph $\bG$ is cellularly embedded, we construct its medial graph $\bG_m$ by placing
a vertex of degree 4 on each edge of  $\bG$ and then drawing the edges of the medial
graph by following the face boundaries of $\bG$.  When $\bG$ is represented as a ribbon graph, it is often easiest to represent it as a cellularly embedded graph and then construct $\bG_m$ as above. 
 If necessary, the representation of $\bG_m$ may subsequently be converted to that of a ribbon graph.

\subsection{Cogs in the context of gems}\label{ss:coggem}
We shall make extensive use of the following description of cogs as cubic graphs equipped with a perfect matching. This description is motivated by the description of embedded graphs as graph-encoded maps, or \emph{gems} (see~\cite{zbMATH00824939,zbMATH03728291} for details on gems).

\begin{definition}
  A \emph{graph-encoded cog}, or \emph{gec}, consists of a graph in which each component is either a cubic graph with a specified perfect matching, or is a free loop. The edges in the perfect matching are called \emph{e-edges} (as they will correspond to the edges of a cog). 
The edges not in the perfect matching are \emph{v-edges} and comprise a disjoint union of cycles (including free loops), called \emph{v-cycles} (as each such cycle will correspond to a vertex in a cog).
The half-edges are therefore divided into \emph{e-half-edges} and \emph{v-half-edges}.
Two gecs are \emph{isomorphic} if there is an isomorphism between their underlying graphs that preserves the sets of e-edges and v-edges.
\end{definition}

\begin{figure}[t!]
\centering
\begin{subfigure}[b]{0.45\textwidth}
\labellist
\small\hair 2pt
\pinlabel $v_1$ at 85 60
\pinlabel $v_2$ at 235 40
\pinlabel $v_3$ at 360 140
\pinlabel $v_4$ at 360 60
\pinlabel $v_5$ at 360 6
\pinlabel $e_1$ at 60 125
\pinlabel $e_2$ at 175 73
\pinlabel $e_3$ at 228 140
\pinlabel $e_4$ at 280 121
\pinlabel $e_5$ at 300 73
\pinlabel $e_6$ at 290 18
\endlabellist
\includegraphics[scale=0.35]{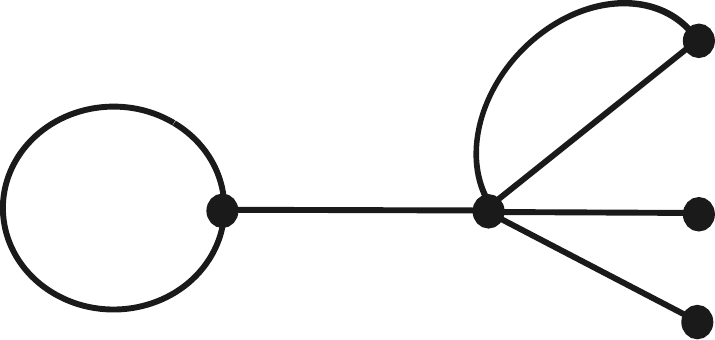}
\vspace{10mm}
\caption{A cog.}
\label{f.cg1}
\end{subfigure}
\hfill
\begin{subfigure}[b]{0.45\textwidth}
\centering
\includegraphics[scale=1.1]{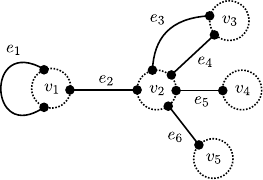}
\caption{A gec. The v-edges are dotted, the e-edges are solid.}
\label{f.cg2}
\end{subfigure}
\caption{A cog with its corresponding gec.} 
\label{fig:cg}
\end{figure}
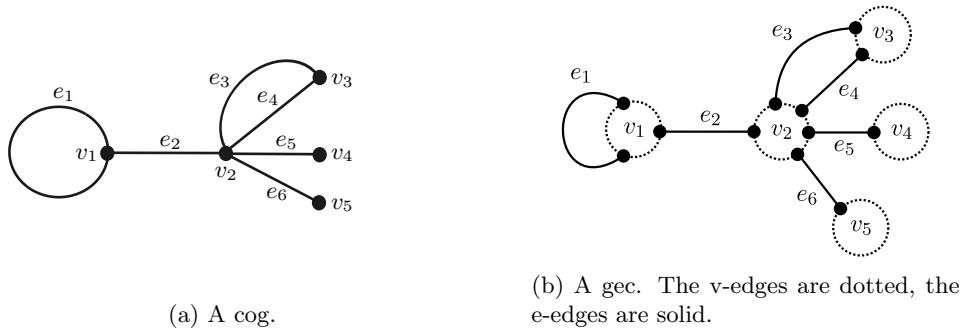

We see in the following proposition that gecs exactly encode cogs. The correspondence is illustrated in Figure~\ref{fig:cg}. Figure \ref{f.cg1} shows a cog, where we use the natural clockwise (or anticlockwise) direction to give the ordering at each vertex, and Figure \ref{f.cg2} shows the corresponding gec.
In figures of gecs,  we shall use solid lines to depict e-edges and dotted lines to depict v-edges. 

\begin{proposition}\label{prop:zxc}
There is a 1-1 correspondence between the set of cogs (up to isomorphism) and the set of gecs (up to isomorphism).
\end{proposition}

\def\buco{\eta}

\begin{proof}

Observe that an undirected cyclic ordering of a finite set $S$ is equivalent to a bijection between $S$ and the vertices in an undirected cycle.

Given a cog, we find a gec by blowing each vertex up into a cycle.  For each vertex $v$ in the cog let $H_v$ be the set of half-edges incident with $v$; the undirected cyclic ordering of $H_v$ gives a bijection $\buco_v$ from $H_v$ to the vertex set of an undirected cycle $C_v$.
If $H_v = \emptyset$ (so $v$ is isolated), $C_v$ is a free loop.
Form the gec whose v-cycles are the cycles $C_v$, whose e-half-edges are the half-edges of the cog with the same pairing into edges, and with each half-edge $h \in H_v$ being incident in the gec with $\buco_v(h) \in V(C_v)$.

Given a gec, we contract the v-cycles to obtain a cog.  For each v-cycle $C$ let $H_C$ be the set of e-half-edges incident with a vertex of $C$; incidence gives a bijection $\buco_C$ from $H_C$ to $V(C)$.  If we contract each v-cycle (including each free loop) to a vertex $v_C$, then the set of half-edges incident with $v_C$ is $H_C$, and $\buco_C$ gives an undirected cyclic ordering of $H_C$, so we have a cog.

These constructions are easily seen to be inverses, proving the result.
\end{proof}

We shall call the correspondence used in the proof of Proposition~\ref{prop:zxc} the \emph{natural correspondence} between cogs and gecs, and talk about the \emph{corresponding gec} for a cog or \emph{corresponding cog} for a gec.

\begin{remark} Readers familiar with gems may note that the gec corresponding to the underlying cog of a graph embedding can be obtained from the gem of the embedding.  This is a cubic graph whose edges are partitioned into three $1$-factors $\ea, \ev, \ef$, where the cycles in $\ea \cup \ev$ represent vertices, the cycles in $\ea \cup \ef$ represent faces, and the cycles in $\ev \cup \ef$ are $4$-cycles representing edges.  The gec can be formed from the gem by contracting all edges in $\ev$, and replacing the components of $\ef$, now digons, by single edges. These single edges are the e-edges of the gec, while the components of $\ea$ are the v-cycles of the gec.
\end{remark}

We will also need to use structures that extend the idea of a gec.
A \emph{generalised gec} is a graph, possibly with free loops, whose edges may be divided into \emph{e-edges}, which form a matching, and \emph{v-edges} (including all free loops), which form a subgraph of maximum degree at most $2$.
A component of the spanning subgraph whose edge set consists of all v-edges is a \emph{v-component}, which may be a path (including a single vertex) or a cycle (including a free loop).

Deleting edges and extracting e-edges of a gec yields a generalised gec.  Moreover, every generalised gec can be obtained from some gec by deletion of e-edges and v-edges, and also from some gec by a combination of extraction and deletion of e-edges.

\subsection{Cog invariants}\label{ss:coginv}
We now define the objects we wish to construct and study in the following sections.
\begin{definition}\label{def:coginv}
 A \emph{cog invariant} is a function on cogs that takes the same value on isomorphic cogs. A \emph{cog polynomial} is a polynomial-valued cog invariant.
\end{definition}

By Propositions~\ref{prop:ppcog2} and~\ref{prop:zxc}, functions defined on gecs or on sets of ribbon graphs closed under taking partial Petrials immediately give rise to cog invariants.

\def\cmbb{ }

\section{A cog polynomial via the matching polynomial} \label{sec:matching_polynomial}
In this section  we demonstrate our first strategy for finding cog polynomials: we add structure to a known polynomial invariant of abstract graphs so that it captures cog properties. We apply this strategy here to the matching polynomial.

Following Farrell~\cite{zbMATH03523610}, the \emph{(bivariate) matching polynomial}, $\Match(G;x,y)$, of a loopless graph  $G=(V,E)$ is defined as the generating function 
\[ \Match(G;x,y) = \sum_{i=0}^{\lfloor |V|/2 \rfloor} \alpha_ix^{|V|-2i}y^{i}, \]
where $\alpha_i$ is the number of matchings in $G$ containing exactly $i$ edges. The matching polynomial satisfies a deletion-extraction relation:
\begin{equation}\label{ed:mde}
\Match(G;x,y) =  \begin{cases}
\Match(G \ba e;x,y) +y\, \Match(G\ext e;x,y) &\text{for an edge $e$ of $G$,}\\
x^{k(G)} & \text{if $G$ is edgeless.}
\end{cases}
\end{equation}

Our strategy is to use the natural correspondence between cogs and gecs, and to  turn the matching polynomial into a cog invariant $\calM(\gec{G};x,y)$ by applying the deletion-extraction relation~\eqref{ed:mde} to \emph{only} the e-edges of a gec. 
A minor technical issue to overcome is that if we start with a gec and delete or extract one of its e-edges then the result generally will not be a gec, and so we take the domain to be generalised gecs instead. 

As we shall soon see, $\calM(\gec{G};x,y)$ encodes information about v-cycles that are saturated by sets of e-edges.  For this reason, we call $\calM(\gec{G};x,y)$ the \emph{saturation polynomial} of a cog.

\begin{definition}\label{def:mc}
Let $\gec{G}=(V,E)$ be a generalised gec and $\mathcal{E}\subset E$ be its set of e-edges. Then the \emph{saturation polynomial}
$\calM(\gec{G};x,y)$ is  defined recursively by
\begin{equation}\label{ed:cde}
\calM(\gec{G};x,y) =  \begin{cases}
\calM(\gec{G} \ba e;x,y) +y\, \calM(\gec{G}\ext e;x,y) &\text{for an e-edge $e\in \mathcal{E}$,}\\
x^{k(\gec{G})} & \text{if $\mathcal{E}=\emptyset$.}
\end{cases}
\end{equation}
\end{definition}
This definition does not a priori give a well-defined function on gecs  as it potentially depends upon the order of edges to which it is applied. However, we shall shortly show that it is independent of this choice and thus does indeed
give a well-defined invariant.

\begin{figure}
\labellist
\small\hair 2pt
\pinlabel $y$ at 395 280
\pinlabel $y$ at 221 145
\pinlabel $y$ at 515 145
\pinlabel $y$ at 102 55
\pinlabel $y$ at 250 55
\pinlabel $y$ at 416 55
\pinlabel $y$ at 560 55
\pinlabel $\emptyset$ at 560 15
\pinlabel $x^2$ at 30 -13
\pinlabel $x^3y$ at 110 -13
\pinlabel $x^2y$ at 188 -12
\pinlabel $x^2y^2$ at 260 -12
\pinlabel $x^2y$ at 355 -12
\pinlabel $x^2y^2$ at 425 -12
\pinlabel $x^2y^2$ at 493 -12
\pinlabel $y^3$ at 561 -12
\pinlabel $=x^2+x^3y+2x^2y+3x^2y^2+y^3$ at 295 -55
\endlabellist
\includegraphics[width=8cm]{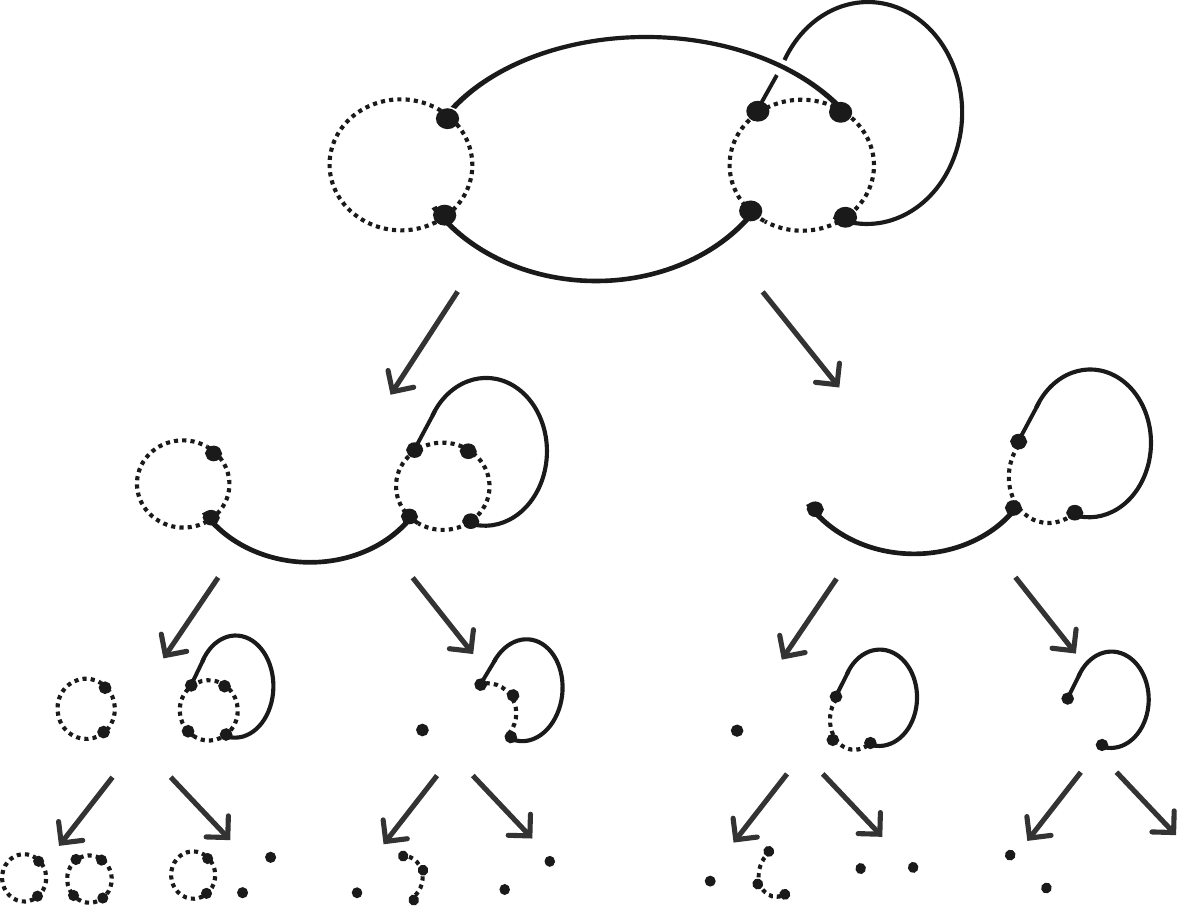}

\vspace{10mm}
\caption{A computation of $\calM(\gec{G};x,y)$.}
\label{fig:eg1}
\end{figure}

Figure~\ref{fig:eg1} shows a computation of $\calM(\gec{G};x,y)$ using Definition~\ref{def:mc}. Example~\ref{ex:M_polynomial}  shows that 
$\calM(\gec{G};x,y)$ can distinguish gecs arising from cogs that have the same underlying graph. Thus, the saturation polynomial is genuinely a cog polynomial rather than just a graph polynomial.
\begin{example}\label{ex:M_polynomial}
Let $\gec{G}_1$ and $\gec{G}_2$ be the gecs shown in Figure~\ref{fig:eg2}. Note that the gecs both correspond to cogs on the same underlying abstract graph (consisting of two vertices, two parallel edges, with one loop). 
Then 
\[ \calM(\gec{G}_1;x,y) 
=x^2+3x^2y+xy^2+2x^2y^2+y^3,\] 
  and
\[ \calM(\gec{G}_2;x,y) 
=x^2+x^3y+2x^2y+3x^2y^2+y^3.\]
\end{example}

\begin{figure}
\labellist
\small\hair 2pt
\pinlabel $\gec{G}_1$ at 115 -15
\pinlabel $\gec{G}_2$ at 410 -15
\endlabellist
\includegraphics[scale=0.6]{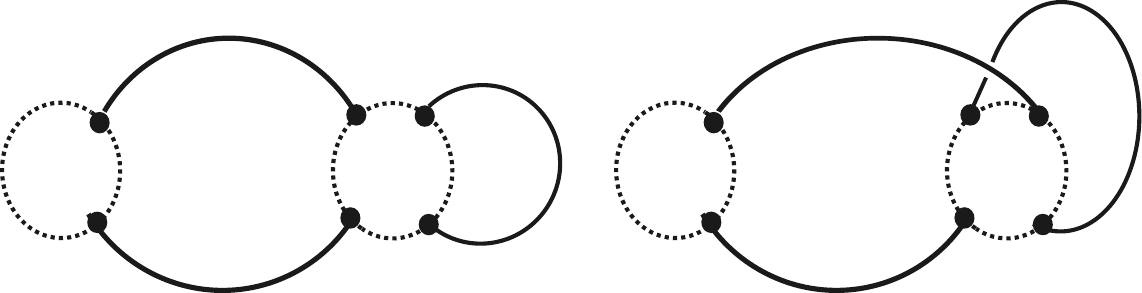}

\vspace{+4mm}
\caption{Two gecs that represent nonisomorphic cogs.}
\label{fig:eg2}
\end{figure}

\bigskip

We now show that $\calM(\gec{G};x,y)$ is well-defined. For this we record the following immediate consequence of the fact that that two e-edges $e$ and $f$ cannot share a vertex.
\begin{proposition}\label{prop:2dag}
Let $\gec{G}=(V,E)$ be a generalised gec,  $\mathcal{E}$ be its set of e-edges, and $e,f\in \mathcal{E}$. Then 
$(\gec{G}\ext e)\ext f = (\gec{G}\ext f)\ext e$ and 
$(\gec{G}\ext e)\ba f = (\gec{G}\ba f)\ext e$. 
\end{proposition}

In light of Proposition~\ref{prop:2dag} we can make the following definition. Let $\gec{G}=(V,E)$ be a generalised gec, let $\mathcal{E}$ be its set of e-edges, and suppose $\mathcal{A}\subseteq \mathcal{E}$. Then $\gec{G}\ext \mathcal{A}$ is the result of extracting all of the edges in $\mathcal{A}$ from $\gec{G}$ in any order.

The following theorem expresses $\calM(\gec{G};x,y)$ in a form that is independent of the order of edge deletion and extraction and this shows the saturation polynomial is well-defined.

\begin{theorem}\label{thm:mwd}
Let $\gec{G}=(V,E)$ be a generalised gec and $\mathcal{E}$ be its set of e-edges. Then
\[
\calM(\gec{G};x,y) 
= \sum_{\mathcal{A} \subseteq \mathcal{E}} y^{|\mathcal{A}|}x^{k(\gec{G}\ext \mathcal{A} \setminus (\mathcal{E}-\mathcal{A}))}
= \sum_{\mathcal{A} \subseteq \mathcal{E}} y^{|\mathcal{A}|}x^{k_{\textup{v}}(\gec{G}\ext \mathcal{A} )},
\]
where $k_{\textup{v}}(\gec{G})$ means the number of v-components of a generalised gec $\gec{G}$.
\end{theorem}

\begin{proof}
Assign an arbitrary linear order to the edges of $\gec{G}$ and compute $\calM(\gec{G};x,y)$ by applying Equation~\eqref{ed:cde} to e-edges with respect to this order. The resulting summands  of $\calM(\gec{G};x,y)$ (before collecting terms) are in bijection with the subsets $\mathcal{A}\subseteq \E$ where $\mathcal{A}$ is the set of e-edges that were extracted to obtain the summand (and the e-edges in $\E-\mathcal{A}$ were deleted). By Proposition~\ref{prop:2dag} we can first extract all the edges of $\mathcal{A}$, recording a factor of $y$ for each.  We can then delete all the e-edges of $\E-\mathcal{A}$ from $\gec{G}\ext \mathcal{A} $  and record a factor of $x$ for each of the resulting components. 
 Thus, each subset $\mathcal{A}$ of the e-edges gives rise to a term $y^{|\mathcal{A}|}x^{k(\gec{G}\ext \mathcal{A} \setminus (\mathcal{E}-\mathcal{A}))}$ and the result follows.
\end{proof}

\medskip

By  Proposition~\ref{prop:zxc}, the invariant  $\calM$ gives rise to a cog invariant: describe a cog as a gec and compute its saturation polynomial.   
We will now express this invariant  directly in terms of cogs. 
 For this let $\cog{G}=(V,E)$ be a cog, $B$ a subset of its edges, and $v$ a vertex. Consider the unordered cyclic sequence $\sigma_v$ of half-edges incident to $v$. Each half-edge is either part of an edge in $B$ or is not. The half-edges from edges in $B$ give subsequences of $\sigma_v$. We call such a subsequence a \emph{$B$-segment} at $v$. So a $B$-segment is a maximal subsequence of $\sigma_v$ consisting of edges in $B$. (The empty sequence is not a $B$-segment.) 
 We let $\seg(v,B)$ denote the number of $B$-segments at $v$. 
 Finally we use $\seg(B)$ to denote the total number of $B$-segments, i.e., $\seg(B) = \sum_{v\in V} \seg(v,B)$. 

 For example, for the cog in Figure~\ref{f.cg1} and with $B=\{e_2, e_4,e_6 \}$, we have $\seg(v_1,B)=\seg(v_3,B)=\seg(v_5,B)=1$, $\seg(v_4,B)=0$, $\seg(v_2,B)=2$, and $\seg(B)=5$.

  Note that when some, but not all, the half-edges at $v$ are in $B$ then $\seg(v,B)$ also equals the number of maximal subsequences of $\sigma_v$ consisting of edges not in $B$. 
However, when there are either no half-edges of $B$ at $v$ or they are all in $B$ this is no longer true. Indeed, $\seg(v,\emptyset)=0$ and $\seg(v,E)=1$. 
  Additionally, as long as not every edge is in $B$, then $\seg(v,B)$  equals the number of times we change between a half-edge
in $B$ and a half-edge not in $B$ as we follow the unordered cyclic sequence $\sigma_v$, but this is not true when every edge is in $B$. A similar observation holds if we count number of times we change between a half-edge not in $B$ to a half-edge  in $B$.

The following theorem describes $ \calM(\gec{G};x, y)$  in terms of cogs.
\begin{theorem}\label{thm:rda}
Let $\cog{G}=(V,E)$ be a cog and $\gec{G}$ be its corresponding gec. Then
\[
 \calM(\gec{G};x, y) = \sum_{B\subseteq E}   y^{|E|-|B|} x^{\seg(B)+\iota(\cog{G})}, 
 \]
where $\seg(B)$ is the number of $B$-segments in the cog $\cog{G}$, and $\iota(\cog{G})$ the number of isolated vertices in $\cog{G}$.
\end{theorem}
\begin{proof}
Let $\mathcal{E}$ be the set of e-edges of $\gec{G}$. 
By replacing $\mathcal{A}$ in Theorem~\ref{thm:mwd} by $\mathcal{B} = \mathcal{E}-\mathcal{A}$ and using Proposition~\ref{prop:2dag} we can write 
\[
\calM(\gec{G};x,y) = \sum_{\mathcal{B} \subseteq \mathcal{E}} y^{|\mathcal{E}|-|\mathcal{B}|}x^{k(\gec{G}\ba \mathcal{B} \ext (\mathcal{E}-\mathcal{B}))}.
\]

Fix some $\mathcal{B} \subseteq \mathcal{E}$ and let $B\subseteq E$ be the corresponding set of edges in the cog $\cog{G}$. 
The gec
$\gec{G}\ba \mathcal{B} \ext (\mathcal{E}-\mathcal{B})$ can be obtained by taking the union of all of the v-cycles in $\gec{G}$ then deleting each vertex that was not an end of an edge in $\mathcal{B}$. 
Thus, the contribution of a cycle of v-edges $C_v$ to $k(\gec{G}\ba \mathcal{B} \ext (\mathcal{E}-\mathcal{B}))$ can be obtained by deleting all the vertices in $C_v$ that were not an end of an edge in $\mathcal{B}$. When $C_v$ is not a free loop, this is easily seen to equal $\seg(v,B)$ where $v$ is the vertex in the cog $\cog{G}$ corresponding to $C_v$. 
When $C_v$ is a free loop it contributes one to the number of connected components, and the corresponding vertex $v$ in $\cog{G}$ is isolated.
The result follows.
\end{proof}

When $\cog{G}$ is a 3-regular or 4-regular cog, Theorem~\ref{thm:rda} reduces to a particularly accessible form.
 For these, if $\cog{G}=(V,E)$ is a cog or graph and $A$ is a subset of its edges, then we say a vertex $v$ of $\cog{G}$ is \emph{saturated} by $A$ if all of its incident edges are in $A$. 
\begin{corollary}\label{cor:fgh}
Let $G$ be a 3-regular graph on $n$ vertices and let $\gec{G}$ be the gec corresponding to its unique cog. Then
\[    \calM(\gec{G};x, y) = \sum_{A\subseteq E(G)} y^{|A|} x^{n-\sat(A)} ,\] 
where $\sat(A)$ is the number of vertices in $G$  saturated by $A$. 
\end{corollary}
\begin{proof}
As $G$ is 3-regular it has no isolated vertices, so by making use of Theorem~\ref{thm:rda} we can write
\[ \calM(\gec{G};x, y) = \sum_{B\subseteq E}   y^{|E|-|B|} x^{\seg(B)} = \sum_{A\subseteq E}   y^{|A|} x^{\seg(E-A)}. \]
The result then follows upon noting that for each vertex $v$ of degree 3, $\seg(v, E-A)$ is $ 0=1-1$ if $v$ is saturated by $A$, and equals $1=1-0$ otherwise. 
\end{proof}

The following corollary is immediate from the previous one.
\begin{corollary}  Let $G$ be a 3-regular graph on $n$ vertices and let  $\gec{G}$ be the gec corresponding to its unique cog.
Then $\calM$ encodes a generating function for the spanning subgraphs of $G$ with a fixed number of vertices of degree 3:
\begin{equation}
    x^n\calM(\gec{G};x^{-1}, 1) = \sum_i{ c_i x^i},
\end{equation}
  where $c_i$ is the number of spanning subgraphs of $C$ that have exactly $i$ of their vertices of degree three.  
\end{corollary}

 We say a vertex $v$ of a 4-regular cog $\cog{G}$ is a \emph{crossing vertex} with respect to a subset of its edges $A$ if as we travel round the vertex with respect to the undirected cyclic order the edges alternate between being in $A$ and not being in $A$.

\begin{corollary}
Let $\cog{G}$ be a 4-regular cog on $n$ vertices and let $\gec{G}$ be its corresponding gec. Then 
\[    \calM(\gec{G};x, y) = \sum_{A\subseteq E(\cog{G})} y^{|A|} x^{n+\cro(A)-\sat(A)} ,\] 
where $\sat(A)$ is the number of vertices in $\cog{G}$ saturated by $A$, and  $cr(A)$ the number of crossing vertices in $\cog{G}$ with respect to $A$.
\end{corollary}
\begin{proof}
Follow the proof of Corollary~\ref{cor:fgh} but use the fact that for each vertex $v$ of degree 4, $\seg(v, E-A) = 1-0-1$ if $v$ is saturated by $A$,  equals $1+1-0$ if $v$ is crossing with respect to $A$, and equals $1+0-0$ otherwise.
\end{proof}

Theorem~\ref{thm:rda} immediately offers a way to define $\calM(\gec{G};x, y)$ directly on a cog $\cog{G}$ rather than through its corresponding gec by defining \[M(\cog{G};x, y) = \sum_{B\subseteq E}   y^{|E|-|B|} x^{\seg(B)+\iota(\cog{G})}.\] However, to obtain a polynomial that satisfies a linear recursion relation that is the cog-version of Equation~\eqref{ed:cde} we define a slightly more general version of this polynomial.   

\begin{definition}
Let $\cog{G}=(V,E)$ be a cog and let $D$ and $X$ be disjoint subsets of $E$. We define
\[M(\cog{G},D,X;x, y) = \sum_{B\subseteq E-(D\cup X)}   y^{|E|-|B\cup D\cup X|} x^{\seg(B\cup D)+\iota(\cog{G})},\]
where $\seg(B\cup D)$ is the number of $(B\cup D)$-segments in the cog $\cog{G}$, and $\iota(\cog{G})$ its number of isolated vertices.
\end{definition}

By  Theorem~\ref{thm:rda}, $M(\cog{G},\emptyset,\emptyset;x, y) = \calM(\gec{G} ; x,y)$
 for  a cog $\cog{G}$ and its corresponding gec $\gec{G}$. In fact a stronger result holds.

\begin{theorem}\label{thm:scg}
Let $\cog{G}=(V,E)$ be a cog and $\gec{G}$ be its corresponding gec. In addition let $D$ and $X$ be disjoint subsets of $E$, and let $\gec{D}$ and $\gec{X}$ be the corresponding sets of $e$-edges of $\gec{G}$.  Then
\[ M(\cog{G},D,X;x, y) =  \calM(\gec{G}\ba \gec{D}\ext \gec{X}  ; x,y).\]
\end{theorem}
\begin{proof}
The proof follows that of Theorem~\ref{thm:rda}.
Let $G$, $\gec{G}$,  $D$, $X$, $\gec{D}$ and $\gec{X}$ be as in the theorem statement, and let $\mathcal{E}$ be the set of e-edges in $\gec{G}$, so $\mathcal{E}-(\gec{D}\cup \gec{X})$ is the set of e-edges in $\gec{G}\ba \gec{D}\ext \gec{X}$. 
Reindexing the sum in Theorem~\ref{thm:mwd} and using Proposition~\ref{prop:2dag} gives  
\begin{align*}
\calM(\gec{G}\ba \gec{D}\ext \gec{X};x,y) 
&= \sum_{\mathcal{B} \subseteq \mathcal{E}-(\gec{D}\cup \gec{X})} y^{|\mathcal{E}-(\gec{D}\cup \gec{X})|-|\mathcal{B}|}x^{k((\gec{G}\ba \gec{D}\ext \gec{X})\ba \mathcal{B} \ext (\mathcal{E}-(\gec{D}\cup \gec{X})-\mathcal{B}))}
\\
&=
\sum_{\mathcal{B} \subseteq \mathcal{E}-(\gec{D}\cup \gec{X})} 
y^{|\mathcal{E}-(\mathcal{B}\cup \gec{D}\cup \gec{X})|}
x^{k(\gec{G}\ba (\gec{D}\cup \gec{B}) \ext (\gec{E}-(\gec{D}\cup \gec{B})))}
.
\end{align*}

Fix some $\mathcal{B} \subseteq \mathcal{E}-(\gec{D}\cup \gec{X})$ and let $B\subseteq E$ be the corresponding set of edges in the cog $\cog{G}$. 
The gec
$\gec{G}\ba (\gec{D}\cup\mathcal{B}) \ext (\mathcal{E}-(\gec{D}\cup\mathcal{B}))$ can be obtained by taking the union of all of the v-cycles in $\gec{G}$ then deleting each vertex that was not an end of an edge in $\gec{D}\cup\mathcal{B}$. 
Thus, the contribution of a cycle of v-edges $C_v$ to $k(\gec{G}\ba (\gec{D}\cup\mathcal{B}) \ext (\mathcal{E}-(\gec{D}\cup\mathcal{B})))$ can be obtained by deleting all the vertices in $C_v$ that were not an end of an edge in $\gec{D}\cup\mathcal{B}$. When $C_v$ is not a free loop, this contribution is easily seen to equal $\seg(v,D\cup B)$ where $v$ is the vertex in the cog $\cog{G}$ corresponding to $C_v$. 
When $C_v$ is a free loop it contributes one to the number of connected components, and the free loop corresponds to a unique isolated vertex in
$\cog{G}$.
Thus, 
$k(\gec{G}\ba (\gec{D}\cup\mathcal{B}) \ext (\mathcal{E}-(\gec{D}\cup\mathcal{B})))$ 
is equal to 
$\seg(B\cup D)+\iota(\cog{G})$. 
Additionally,  $|\mathcal{E}-(\mathcal{B}\cup \gec{D}\cup \gec{X})| = |E|-|B\cup D\cup X|$.
Thus, for corresponding sets $\gec{B}$ and $B$ we have  
$y^{|\mathcal{E}-(\mathcal{B}\cup \gec{D}\cup \gec{X})|}
x^{k(\gec{G}\ba (\gec{D}\cup \gec{B}) \ext (\gec{E}-(\gec{D}\cup \gec{B})))}
=
 y^{|E|-|B\cup D\cup X|} x^{\seg(B\cup D)+\iota(\cog{G})}$ and
the result follows.
\end{proof}

The following result gives a recursive definition of $\calM(\gec{G}\ba \gec{D}\ext \gec{X};x,y) $. It follows immediately from Equation~\eqref{ed:cde} and Theorem~\ref{thm:scg}
\begin{corollary}
Let $\cog{G}=(V,E)$ be a cog and $\gec{G}$ be its corresponding gec. In addition let $D$ and $X$ be disjoint subsets of $E$. Then 
\[  M(\cog{G},D,X)  =
\begin{cases}
M(\cog{G},D\cup \{e\},X) +y M(\cog{G},D,X\cup \{e\})  &
\text{for any $e\in E-(D\cup X)$,}\\
x^{\seg(D)+\iota(\cog{G})} & \text{if $D\cup X = E$,}
\end{cases}
\]
where we have written  $M(\cog{G},S,T)$ for $M(\cog{G},S,T;x, y)$. 
\end{corollary}

\section{A cog polynomial via the transition polynomial} \label{sec:trans2}

In the previous section we obtained a cog invariant by regarding a cog as an abstract graph with additional structure (specifically, a gec) and then adapting a known graph polynomial so that its value depends upon this additional structure. An alternative approach is to regard cogs as embedded graphs in which some information is lost, then to obtain a cog invariant by adapting an embedded graph invariant so it no longer depends on the absent information. This section  considers this process of `losing information' starting with the topological transition polynomial to create another new cog invariant.

We first  review some additional operations on ribbon graphs.
Let $\bG$ be a ribbon graph and $e\in E(\bG)$. Then $\bG\ba e$ denotes the ribbon graph obtained from $\bG$ by \emph{deleting} the edge $e$.
If $u$ and $v$ are the (not necessarily distinct) vertices incident with $e$, then $\bG/ e$ denotes the ribbon graph obtained as follows: consider the boundary component(s) of $e\cup\{u,v\}$ as curves on $\bG$. For each resulting curve, attach a disc (which will form a vertex of $\bG/e$) by identifying its boundary component with the curve. Delete $e$, $u$ and $v$ from the resulting complex to get the ribbon graph $\bG/e$. We say $\bG/e$ is obtained from $\bG$ by {\em contracting} $e$. See Table~\ref{tablecontractrg} for the local effect of contracting an edge of a ribbon graph.
For a ribbon graph $\bG$ and edge $e$, we let  $\bG\tcon e$ denote the ribbon graph $\bG^{\tau(e)}/e$.  We call the operation this defines a \emph{Penrose-contraction}. Again, see Table~\ref{tablecontractrg}.

\begin{table}[t]
\centering
\begin{tabular}{|c||c|c|c|}\hline
& non-loop & non-orientable loop & orientable loop \\ \hline
\raisebox{6mm}{$\bG$} 
&\includegraphics[scale=.25]{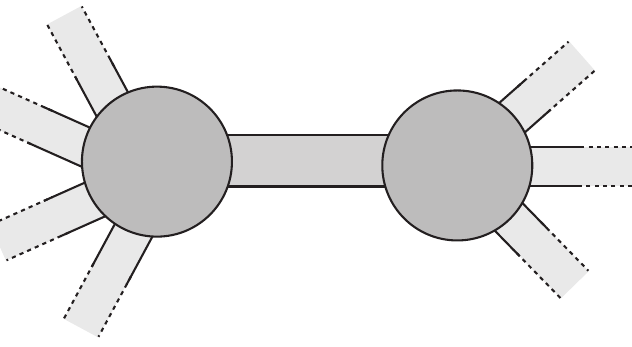} &\includegraphics[scale=.25]{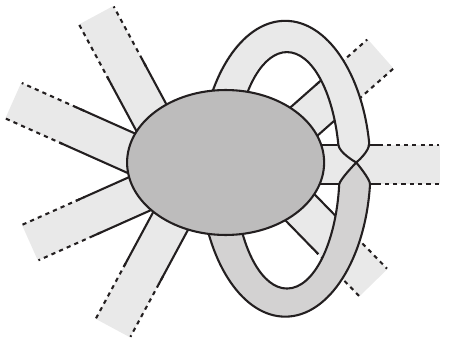} &\includegraphics[scale=.25]{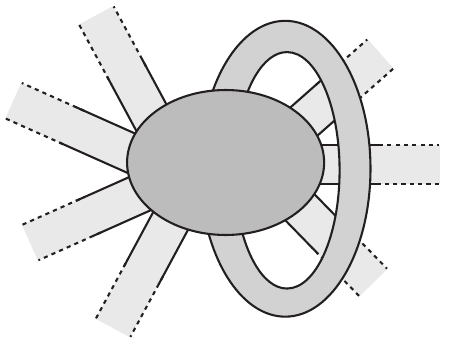}
\\ \hline
\raisebox{6mm}{$\bG/e$} 
&\includegraphics[scale=.25]{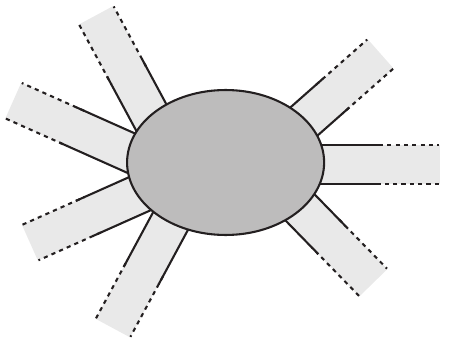} &\includegraphics[scale=.25]{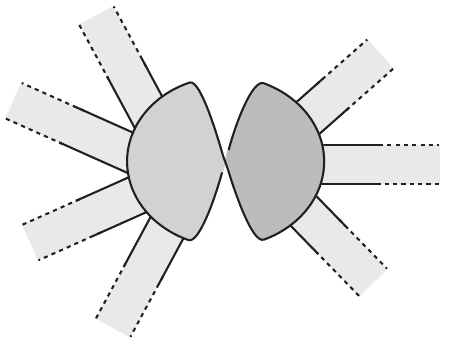}&\includegraphics[scale=.25]{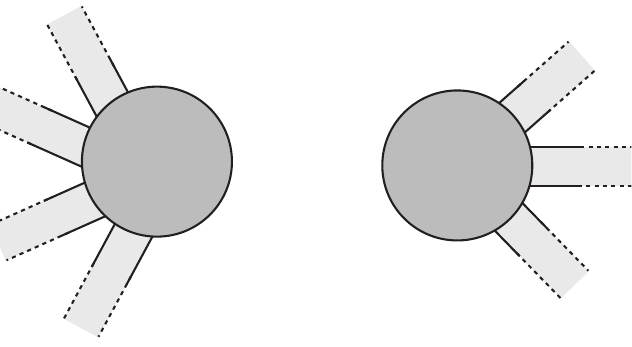}

\\ \hline
\raisebox{6mm}{$\bG\tcon e$} 
 & \includegraphics[scale=.25]{ch4_37a}&\includegraphics[scale=.25]{ch4_38a}&\includegraphics[scale=.25]{ch4_37a}
\\ \hline
\end{tabular}
\caption{Contracting and Penrose-contracting an edge of a ribbon graph.}
\label{tablecontractrg}
\end{table}

\medskip

 The \emph{topological transition polynomial}, introduced in \cite{MR2869185}, is a multivariate polynomial of ribbon graphs (or equivalently, of graphs embedded in surfaces). It contains both the 2-variable version of  Bollob\'as and Riordan's ribbon graph polynomial~\cite{MR1851080,MR1906909} and the Penrose polynomial~\cite{MR1428870,MR2994409} as specialisations, and is closely related to  Jaeger's transition polynomial~\cite{MR1096990} and the 
 generalised transition polynomial of 
 \cite{MR1980048}.

\begin{definition}\label{toptranspoly}
If $\bG=(V,E)$ is a ribbon graph, then the
 \emph{topological transition polynomial}, $Q(\bG; (\alpha,\beta,\gamma), t)$, can be defined as 
\label{Part3} \[Q(\bG; (\alpha,\beta,\gamma), t) =\sum_{(X,Y,Z) \in \mathcal{P}_3(E)}  
 \alpha^{|X|}\beta^{|Y|}\gamma^{|Z|}
 t^{b( \bG^{\tau(Z)}\ba Y)},\] 
where $\mathcal{P}_3(E)$ denotes the set of ordered partitions of $E$ into three blocks (the blocks may be empty), and where $b(\bG^{\tau(Z)}\ba Y)$ denotes the number of boundary components of $\bG^{\tau(Z)}\ba Y$.
\end{definition}
As shown in~\cite[Theorem~5.5]{MR2869185}, the topological transition polynomial can be defined through a recursion relation
\begin{equation}\label{eq:qdc}
Q(\bG) = 
\begin{cases}
\alpha Q(\bG/ e) +\beta Q(\bG\ba e) +\gamma Q(\bG\tcon e) & \text{for an edge $e$,}\\
t^{k(\bG)} & \text{if $\bG$ is edgeless,}
\end{cases} 
\end{equation}
where we have written $Q(\bG)$ for $Q(\bG; (\alpha,\beta,\gamma), t)$,  and where $k(\bG)$ denotes the number of connected components of $\bG$. 
Additional background on the topological transition polynomial can be found in, for example,~\cite{MR3086663}.

It is immediately seen from Equation~\eqref{eq:qdc} that 
\[Q(\bG; (\alpha,\beta,\alpha), t) = Q(\bG^{\tau(e)}; (\alpha,\beta,\alpha), t).\] 
Thus, the value of $Q(\bG; (\alpha,\beta,\alpha), t)$ is constant on the set of all partial Petrie duals of $\bG$, which by Proposition~\ref{prop:ppcog2} is the set of ribbon graphs with the same underlying cog $\cog{G}$ as $\bG$.  This means that $Q(\bG; (\alpha,\beta,\alpha), t)$  is an invariant of the cog $\cog{G}$.
The resulting cog invariant is naturally expressed in terms of gecs as we now describe.

\medskip

We work in the $\mathbb{Z}$-module of formal linear combinations of graphs. 
Let  $G=(V,E)$ be a graph, possibly with free loops, and suppose $e=(u,v)$ is a non-loop edge such that $u$ and $v$ have degree three. 
We define $G\vsplice e$ and $G\hsplice e$, illustrated in Table~\ref{tab:splice}, as follows. 
\begin{itemize}
\item Suppose $e$ is adjacent to four distinct edges $(u_1,u)$, $(u_2,u)$, $(v_1,v)$, and $(v_2,v)$ (it is possible $u_1=u_2$ or $v_1=v_2$ here). Then  $G\vsplice e$ is the graph obtained by adding edges $(u_1,u_2)$, $(v_1,v_2)$ to $G\ba \{u,v\}$. 
Furthermore, $G\hsplice e$ is the sum of the graph obtained by adding edges  $(u_1,v_1)$ and $(u_2,v_2)$ to $G\ba \{u,v\}$ with the graph obtained by adding edges  $(u_1,v_2)$ and $(u_2,v_1)$ to $G\ba \{u,v\}$.

\item Suppose $e$ is incident to exactly one loop, and without loss of generality assume the loop meets $u$. Let  $G'$ be the graph obtained by deleting $u$ and then contracting either one of the two edges incident to $v$.
Then  $G\vsplice e$ is the union of $G'$ and a free loop, and  $G\hsplice e = G'+G'$.

\item Suppose $e$ is incident to two loops. Then $G\vsplice e$ is the union of $G\ba \{u,v\}$ and two free loops, while $G\hsplice e$ is the sum of two copies of the union of $G\ba \{u,v\}$ and one free loop. 

\item Suppose there is exactly one edge $f$  parallel to $e$. Let $G'$ be the graph obtained by deleting $e$ and contracting each of the remaining two edges incident to $f$. Then   $G\vsplice e$ is defined to be $G'$, while   $G\hsplice e$ is the sum of $G'$ with the union of $G'$ and a free loop.

\item Suppose $e$ and its incident edges form three parallel edges. Let $G'$ equal $G\ba \{u,v\}$. Then  $G\vsplice e$ is defined to be the union of $G'$ and one free loop, while $G\hsplice e$ is the sum of the union of $G'$ and two free loops with the union of $G'$ and one free loop

\end{itemize}

In the case where $G=(V,E)$ has a set $\E$ of e-edges and $E-\E$ of v-edges (for example if $G$ is a gec), then we set $\E-\{e\}$ to be the e-edges of $G\vsplice e$ and $G\hsplice e$. The remaining edges are v-edges.

We highlight that the reason we define $G\hsplice e$ to be a sum is that, as there is no cyclic ordering at the vertices of $e$, we cannot distinguish between the two summands, so cannot construct either summand in isolation of the other.

\begin{table}
\centering
\begin{tabular}{|c|c|c|}\hline
$\gec{G}$ & $\gec{G}\vsplice e$ & $\gec{G}\hsplice e$ \\ \hline
\labellist
\large\hair 2pt
\pinlabel $e$ at 110 77
\endlabellist
\includegraphics[scale=0.3]{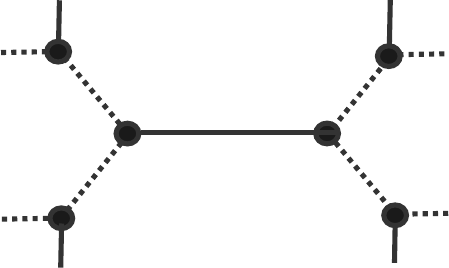}
&\includegraphics[scale=.3]{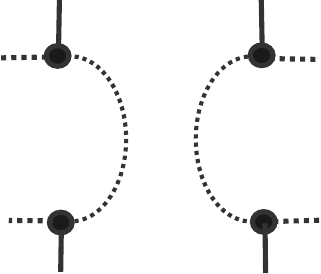} 
&\includegraphics[scale=.30]{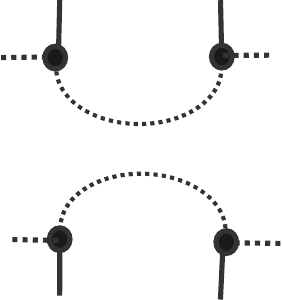} \raisebox{6mm}{$+$} \includegraphics[scale=.30]{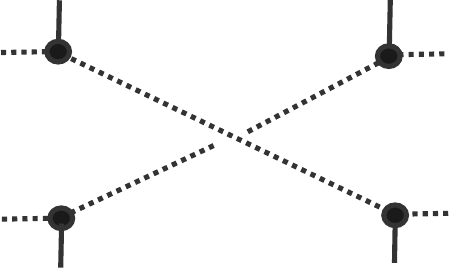}
\\ \hline
\labellist
\large\hair 2pt
\pinlabel $e$ at 120 75
\endlabellist
\includegraphics[scale=.30]{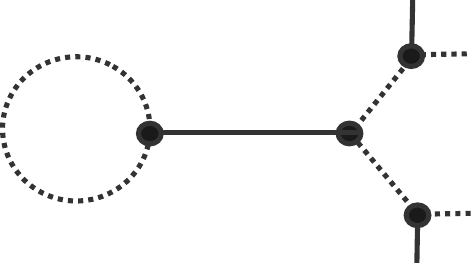}
&\includegraphics[scale=.30]{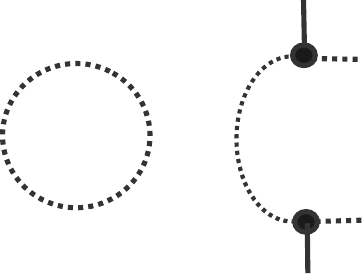} 
&\includegraphics[scale=.30]{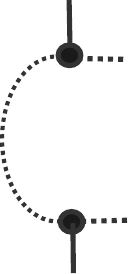} \raisebox{6mm}{$+$} \includegraphics[scale=.30]{SplitSplice7}
\\ \hline
\labellist
\large\hair 2pt
\pinlabel $e$ at 120 50
\endlabellist
\includegraphics[scale=.30]{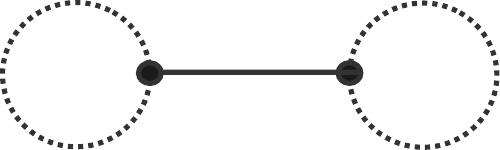}
&\includegraphics[scale=.30]{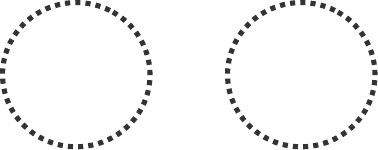} 
&\includegraphics[scale=.30]{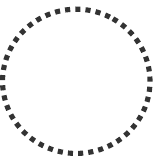} \raisebox{3mm}{$+$} \includegraphics[scale=.30]{SplitSplice10}
\\ \hline
\labellist
\large\hair 2pt
\pinlabel $e$ at 60 42
\endlabellist
\includegraphics[scale=.30]{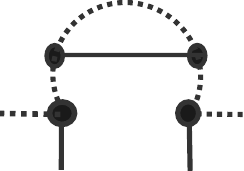}
&\includegraphics[scale=.30]{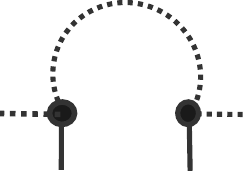} 
& \hspace{-5pt}\includegraphics[scale=.30]
{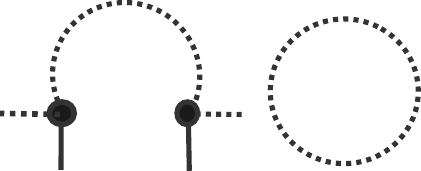} \hspace{2pt}\raisebox{3mm}{$+$} \hspace{1pt}\includegraphics[scale=.30]{SplitSplice12}
\\ \hline
\labellist
\large\hair 2pt
\pinlabel $e$ at 40 50
\endlabellist
\includegraphics[scale=.30]{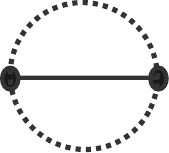}
&\includegraphics[scale=.30]{SplitSplice10} 
& \hspace{-10pt}\includegraphics[scale=.30]{splitsplice9}
\raisebox{3mm}{$+$} \includegraphics[scale=.30]{SplitSplice10}\\
\hline
\end{tabular}

\caption{An illustration of $\gec{G}\protect\vsplice e$ and $\gec{G}\protect\hsplice e$ when $\gec{G}$ is a gec or pointed-gec.}
\label{tab:splice}
\end{table}

\medskip

Suppose $\gec{G}$ is a gec representing a cog $\cog{G}$, and that $\bG$ is a ribbon graph with underlying cog $\cog{G}$.  Then it is clear that 
   $\gec{G} \vsplice e$ corresponds to the class of ribbon graphs represented by $\bG \ba e$, and $\gec{G} \hsplice e$ corresponds to $\bG / e +\bG \tcon  e$,  with $e$ denoting corresponding edges in each case.

Thus, it follows immediately from Equation~\eqref{eq:qdc} that we can define a gec-polynomial $\Q(\gec{G}; (\alpha,\beta),t)\in \mathbb{Z}[\alpha,\beta,t]$ for a gec $\gec{G}$ with set of e-edges $\E$, by setting
\begin{equation}\label{eq:qgeq1}
\Q(\gec{G};(\alpha,\beta),t) = 
\begin{cases}
\alpha\, \Q(\gec{G} \hsplice e;(\alpha,\beta),t) +\beta \, \Q(\gec{G} \vsplice e;(\alpha,\beta),t) & \text{for $e\in \E$}\\
t^{k(\gec{G})} & \text{if $\E=\emptyset$},
\end{cases} 
\end{equation}
and where $\Q(\gec{G};(\alpha,\beta),t)$ is extended linearly over sums of gecs.

Work of Jaeger~\cite{MR1096990} and Aigner~\cite{MR1428870} provides a combinatorial interpretation of the topological transition polynomial as a count of ``$k$-valuations'' which are types of edge colourings of medial graphs. From this, and as we shall see below, we can deduce a combinatorial interpretation of $\Q(\gec{G};(\alpha,\beta),t)$. However, from this combinatorial perspective it makes sense to extend the domain of the topological transition polynomial, and hence work with a generalisation of $\Q(\gec{G};(\alpha,\beta),t)$. The reason for this is that the topological transition polynomial for ribbon graphs does not quite give the generating function of \mbox{$k$-valuations} (with respect to vertex type). However, in~\cite[Theorem~4.1]{zbMATH07021383} it was shown that the generating function for \mbox{$k$-valuations} is given by a version of the transition polynomial for ``edge-point ribbon graphs''. We shall now describe this polynomial and its corresponding cog invariant.  This will let us interpret evaluations of the cog invariant in terms of $k$-valuations. 

Following~\cite{zbMATH07021383}, an \emph{edge-point ribbon graph} is a pair $(\bG,P)$ where $\bG$ is a ribbon graph and $P$ a subset of its edges. (Note two models for edge-point ribbon graph are given in~\cite{zbMATH07021383}; we are using their edge 2-coloured model here.) 
The transition polynomial for an edge-point ribbon graph, $\Omega(\bG,P;w,x,y,t)$,  is defined by 
\begin{equation}\label{eq:qdcep}
\Omega(\bG,P) = 
\begin{cases}
w \,\Omega(\bG,P\cup\{e\}) +x \, \Omega(\bG/e,P-\{e\}) &\\ \qquad +y\, \Omega(\bG\ba e,P-\{e\}) +z \,\Omega(\bG\tcon e,P-\{e\})& \text{for an edge $e\notin P$}\\
t^{k(\bG)} & \text{if $E(\bG)=P$,}
\end{cases} 
\end{equation}
where we have written $\Omega(\bG,P)$ for $\Omega(\bG,P;w,x,y,t)$, etc., and $E(\bG)$ is the edge set of $\bG$. (This is \cite[Definition~2]{zbMATH07021383} expressed in the  present notation.) Note that $\Omega(\bG,\emptyset; 0,\alpha,\beta,\gamma,t) = Q(\bG;(\alpha,\beta,\gamma),t)$.

\medskip

We shall see that edge-point ribbon graphs relate to pointed-gecs. A \emph{pointed-gec} is an edge-labelled graph arising from a gec by contracting some e-edges. See Figure~\ref{fig:pointedgec} for an example. 
Thus, a pointed gec is a graph with vertices of degrees 3 and 4 (only) and a matching that covers exactly the degree 3 vertices. The matching gives the \emph{e-edges}.

\begin{figure}
\centering
  \begin{subfigure}[b]{0.45\textwidth}
\centering
\centering
\includegraphics[scale=1.1]{gec-production_edit}
\caption{Gec.}
\label{f.pg2}
\end{subfigure}
\hfill
\begin{subfigure}[b]{0.45\textwidth}
\includegraphics[scale=1.1]{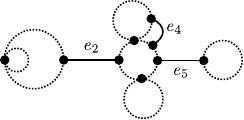}
\caption{Pointed-gec.}
\label{f.pg1}
\end{subfigure}
\caption{A gec and a pointed-gec. The pointed-gec arises by contracting the e-edges $e_1, e_3, e_6$ of the gec.}
\label{fig:pointedgec}
\end{figure}

Switching to pointed-gecs allows us to extend Equation~\eqref{eq:qgeq1} as in the following definition. We shall see that the $\Q(\gec{G} / e)$ term in the definition relates to the $\Omega(\bG,P\cup\{e\})$ term of Equation~\eqref{eq:qdcep}.
\begin{definition}\label{def:qfull}
Let $\gec{G}$ be a pointed-gec with $\E$ its set of e-edges. Then we define $\Q(\gec{G};(\alpha,\beta, \gamma),t)$ recursively through
\begin{equation}\label{eq:qgeq2}
\Q(\gec{G}) = 
\begin{cases}
\alpha\, \Q(\gec{G} \hsplice e) +\beta \, \Q(\gec{G} \vsplice e) +\gamma \, \Q(\gec{G} / e) & \text{for $e\in \E$}\\
t^{k(\gec{G})} & \text{if $\E=\emptyset$},
\end{cases} 
\end{equation}
where we have written $Q(\gec{G})$ for $\Q(\gec{G};(\alpha,\beta, \gamma),t)$, etc.
\end{definition}
This polynomial is not a priori well-defined. Theorem~\ref{thm:qss} below verifies that it is. 

As a simple example, the gec $\gec{G}$ consisting of a single e-edge and single v-cycle of length two has   $\Q(\gec{G};(\alpha,\beta, \gamma),t)= \alpha t^2 +(\alpha+\beta+\gamma)t$.

We will make use of the following proposition. As the result is easily checked, we omit the proof.
\begin{proposition} \label{spliceorder}
Let $\gec{G}$ be a pointed-gec, $\E$ its set of e-edges, with $e,f, g$ distinct elements of $\E$. Then 
$(\gec{G}\hsplice e) \hsplice f = (\gec{G}\hsplice f) \hsplice e$,
$(\gec{G}\vsplice e) \vsplice f = (\gec{G}\vsplice f) \vsplice e$,
and 
$(\gec{G} / e) / f = (\gec{G} / f) / e$.
Furthermore, the operations $\hsplice e$, $\vsplice f$, and $/ g$ on $\gec{G}$ can be carried out in any order without changing the result.
\end{proposition}

In light of Proposition~\ref{spliceorder}, if $\gec{A}$ is a subset of e-edges of a gec $\gec{G}$ then we can define 
$\gec{G}\hsplice \gec{A}$, $\gec{G}\vsplice \gec{A}$, and $\gec{G}/ \gec{A}$ to be the result of applying the relevant operations to each edge in $\gec{A}$ in any order. Note that $\gec{G}\hsplice \gec{A}$ is a sum of $2^{|\gec{A}|}$ pointed-gecs.

\begin{theorem}\label{thm:qss}
Let $\gec{G}$ be a pointed-gec with $\E$ its set of e-edges. Then 
\[ \Q(\gec{G};(\alpha,\beta,\gamma),t) = \sum_{(\gec{X},\gec{Y},\gec{Z})\in\mathcal{P}_3(\E)} 
\alpha^{|\gec{X}|} \beta^{|\gec{Y}|} \gamma^{|\gec{Z}|}
t^{k( ((\gec{G} \vsplice \gec{Y}) / \gec{Z}) \hsplice \gec{X} )},
\]
where $k$ is extended linearly over the sum, and $\mathcal{P}_3(\E)$ denotes the set of ordered partitions of $\E$ into three blocks (the blocks may be empty).
\end{theorem}
\begin{proof}
The proof is routine and follows that of Theorem~\ref{thm:mwd}. 
Assign an arbitrary linear order to the edges of $\gec{G}$ and compute $\Q(\gec{G};(\alpha,\beta,\gamma),t)$  by applying Equation~\eqref{eq:qgeq2} to e-edges with respect to this order. The resulting summands  of $\Q(\gec{G};(\alpha,\beta,\gamma),t)$ (before collecting terms) are in bijection with the partitions $(\gec{X},\gec{Y},\gec{Z})\in\mathcal{P}_3(\E)$
where the operation ``\,$\hsplice$\,''   is applied to the e-edges in $\gec{X}$,  ``$\,\vsplice$\,'' to those in $\gec{Y}$, and ``\,/\,'' to those in $\gec{Z}$.
By Proposition~\ref{spliceorder} we can carry out these operations in any order and so each summand contributes $\alpha^{|\gec{X}|} \beta^{|\gec{Y}|} \gamma^{|\gec{Z}|}
t^{k( ((\gec{G} \vsplice \gec{Y}) / \gec{Z}) \hsplice \gec{X} )}$ to $\Q(\gec{G};(\alpha,\beta,\gamma),t)$. The result follows. 
\end{proof}

We shall now relate the polynomials $Q(\gec{G})$ and $\Omega(\bG,P)$.
Let $\gec{G}$ be a pointed-gec. A (non-unique) pointed ribbon graph can be obtained from $\gec{G}$ as follows.  
First construct a gec $\gec{G}'$  by ``splitting'' open each 4-valent vertex into two 2-valent vertices and add an e-edge between these degree two vertices, as indicated in Figure~\ref{fig:sp}.  Note that $\gec{G}'$ is not uniquely constructed as there are three ways to split each degree 4-vertex. Next chose any ribbon graph $\bG$ that represents  $\gec{G}'$ via the natural correspondences in Propositions~\ref{prop:ppcog2} and~\ref{prop:zxc}. The edge-pointed ribbon graph is $(\bG,P)$ where $P$ is the set of edges that arose from the degree four vertices of $\gec{G}$. We say that $(\bG,P)$ is an edge-pointed ribbon graph \emph{representing} $\gec{G}$. 

\begin{figure}
 \includegraphics[scale=.25]{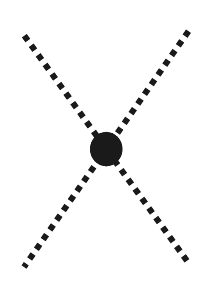}  
\qquad \raisebox{5mm}{$\longrightarrow$}\qquad
  \includegraphics[scale=.25]{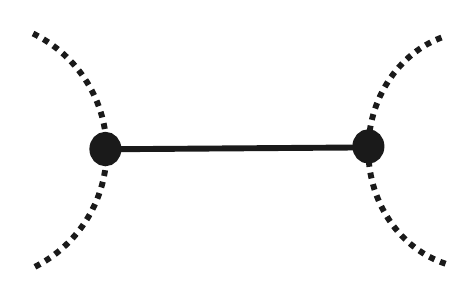}  
\caption{Constructing $\gec{G}'$ from $\gec{G}$.}
\label{fig:sp}
\end{figure}

Observe that $\gec{G}$ can be recovered from $(\bG,P)$ by constructing a gec representing $\bG$ (again via the natural correspondences of Propositions~\ref{prop:ppcog2} and~\ref{prop:zxc}), then 
contracting all the edges in that gec that correspond to edges in $P$.

It is easy to see that if $(\bG,P)$ is an edge-pointed ribbon graph \emph{representing} $\gec{G}$, then 
$(\bG,P\cup\{e\})$ represents $\gec{G}/e$, while 
$(\bG/e,P-\{e\})+(\bG\tcon e,P-\{e\})$ represents $\gec{G}\hsplice e$,
and 
 $(\bG\ba e,P-\{e\})$ represents $\gec{G}\vsplice e$.
These observations allow us to prove the following result.

\begin{lemma}\label{lem:dhg}
Let $\gec{G}$ be a pointed-gec and $(\bG,P)$ be an edge-pointed ribbon graph representing it. Then
\[ \Q(\gec{G};(\alpha,\beta,\gamma),t) =
\Omega(\bG,P; \gamma, \alpha, \beta,\alpha, t).
\]
\end{lemma}
\begin{proof}
We prove the result by induction on the number of e-edges of $\gec{G}$. When $\gec{G}$ has no e-edges then $P$ is the edge set of $\bG$. The base case for the induction follows. Now suppose that  $\gec{G}$ has some e-edges and the result holds for all pointed-gecs with fewer e-edges. Then
\begin{align*}
\Q(\gec{G};(\alpha,\beta,\gamma),t) &=
\alpha\, \Q(\gec{G} \hsplice e;(\alpha,\beta,\gamma),t) +\beta \, \Q(\gec{G} \vsplice e;(\alpha,\beta,\gamma),t) +\gamma \, \Q(\gec{G} / e;(\alpha,\beta,\gamma),t)
\\&=
\alpha\, \Omega(\bG/e,P-\{e\};\gamma, \alpha, \beta,\alpha, t)+ \alpha\, \Omega(\bG\tcon e,P-\{e\};\gamma, \alpha, \beta,\alpha, t) 
\\&\qquad
+\beta \,\Omega(\bG\ba e,P-\{e\};\gamma, \alpha, \beta,\alpha, t)
+\gamma \,\Omega(\bG,P\cup\{e\};\gamma, \alpha, \beta,\alpha, t)
\\&= \Omega(\bG,P;\gamma, \alpha, \beta,\alpha, t),
\end{align*}
and the result follows by induction.
\end{proof}

\medskip

Next we show that $\Q(\gec{G};(\alpha,\beta,\gamma),t)$ counts certain colourings of a gec.
Let $\gec{G}$ be a pointed-gec, $\E$ its set of e-edges, and $\V$ its set of v-edges. Let $k\in \mathbb{N}$. Then a \emph{$k$-valuation} on $\gec{G}$ is a colouring  $\phi:\V \rightarrow [k]$ such that each  e-edge is incident to an even number (possibly zero) of v-edges of each colour, and all v-edges incident to a degree four vertex are of the same colour.

For an e-edge $e=(u,v)\in \E$ suppose $g_u$ and $f_u$ are the (not necessarily distinct) v-edges incident to $u$, and $g_v$ and $f_v$ are the (not necessarily distinct) v-edges incident to $v$. 
As shown in Table~\ref{tablekvaluationtypes}, we say a $k$-valuation is of type \emph{total} at $e$ if $\phi(g_u)=\phi(f_u)=\phi(g_v)=\phi(f_v)$; it is of type \emph{split} at $e$ if $\phi(g_u)=\phi(f_u)\neq\phi(g_v)=\phi(f_v)$; and is of type \emph{spliced} at $e$ if either $\phi(g_u)=\phi(g_v)\neq\phi(f_u)=\phi(f_v)$, or 
$\phi(g_u)=\phi(f_v)\neq\phi(f_u)=\phi(g_v)$.

\begin{table}[t]
\centering
\begin{tabular}{|c|c|c|c|}
\hline
\quad
\labellist
\small\hair 2pt
\pinlabel $e$ at 110 88
\pinlabel $u$ at 20 68
\pinlabel $v$ at 198 68
\pinlabel $g_u$ at -5 125
\pinlabel $f_u$ at -5 11
\pinlabel $g_v$ at 233 125
\pinlabel $f_v$ at 233 11
\endlabellist
 \includegraphics[scale=.25]{k-valuation}  
 \quad
 &
 \labellist
\small\hair 2pt
\pinlabel $i$ at 57 110
\pinlabel $i$ at 170 110
\pinlabel $i$ at 57 25
\pinlabel $i$ at 170 25
\endlabellist
 \includegraphics[scale=.25]{k-valuation}  
&
 \labellist
\small\hair 2pt
\pinlabel $i$ at 57 110
\pinlabel $j$ at 170 110
\pinlabel $i$ at 57 25
\pinlabel $j$ at 170 25
\endlabellist
 \includegraphics[scale=.25]{k-valuation}  
&
 \labellist
\small\hair 2pt
\pinlabel $i$ at 57 110
\pinlabel $i$ at 170 110
\pinlabel $j$ at 57 25
\pinlabel $j$ at 170 25
\endlabellist
 \includegraphics[scale=.25]{k-valuation}  
 \quad
  \labellist
\small\hair 2pt
\pinlabel $i$ at 57 110
\pinlabel $j$ at 170 110
\pinlabel $j$ at 57 25
\pinlabel $i$ at 170 25
\endlabellist
 \includegraphics[scale=.25]{k-valuation}  
\\
\hline
e-edge $e$ & total & split & spliced
\\
\hline

\end{tabular}
\caption{$k$-valuation types on an e-edge where $i\neq j$.}
\label{tablekvaluationtypes}
\end{table}

\begin{theorem} \label{thm:kval}
Let $\gec{G}$ be a pointed-gec,  $\E$ be its set of e-edges and $k \in \mathbb{N}$. Then
\[\gec{Q}(\gec{G};(\alpha,\beta,\gamma),k)= 
\sum_{\substack{\phi \\ \text{ a $k$-valuation} \\ \text{of } \gec{G}}}
(2\alpha+\beta+\gamma)^{\ktotal(\phi)} \alpha^{\ksplice(\phi)} \beta^{\ksplit(\phi)},\]
where $\ktotal(\phi)$ is the number of total e-edges  in the $k$-valuation $\phi$, $\ksplit(\phi)$ is the number of split e-edges  in $\phi$, 
$\ksplice(\phi)$ is the number of spliced e-edges in $\phi$.
\end{theorem}
\begin{proof}[Sketch of proof.]
Let $\gec{G}$ be a pointed-gec and  $(\bG,P)$ be any edge-pointed ribbon graph representing it. Let $P_{\gec{G}}$ be the set of degree four vertices in $\gec{G}$.
Contracting all of the e-edges of  $\gec{G}$ results in a 4-regular graph $F$. It is clear that  $F$ is the medial graph of $\bG$. Moreover, the vertices of $F$ that came from those in $P_{\gec{G}}$ correspond to the edges in $P$.
It is also clear that a $k$-valuation of $\gec{G}$ corresponds to a $k$-valuation of $F$  in which all the vertices in $P$ are total. (See~\cite[Section~1 and Definition~4.1]{zbMATH07021383} for the definition of medial graphs, $k$-valuations, and related terms.)
Thus we have
\begin{align*}
Q(\gec{G};(\alpha,\beta,\gamma),k) &= \Omega(\bG,P; \gamma, \alpha, \beta,\alpha, t)
\\&=
\sum_{\substack{\phi \\ \text{ a $k$-valuation} \\ \text{of } F}}  (2\alpha+\beta+\gamma)^{\mathrm{tot}(\phi)} \alpha^{\mathrm{wh}(\phi)+\mathrm{cr}(\phi)} \beta^{\mathrm{bl}(\phi)}
\\&=
\sum_{\substack{\phi \\ \text{ a $k$-valuation} \\ \text{of } \gec{G}}}
(2\alpha+\beta+\gamma)^{\ktotal(\phi)} \alpha^{\ksplice(\phi)} \beta^{\ksplit(\phi)},
\end{align*}
where the first equality is from Lemma~\ref{lem:dhg}, the second is~\cite[Theorem~4.2]{zbMATH07021383}, and the last is from the above mentioned correspondence between $k$-valuations of $\gec{G}$ and $F$.
\end{proof}

The polynomials $\Q(\gec{G})$ of this section and $\calM(\gec{G})$ of the previous section are not equivalent invariants. For example, $\calM(\gec{G})$ will distinguish the two gecs shown in Figure~\ref{fig.mq}, but $\Q(\gec{G})$ does not.  On the other hand, $\Q(\gec{G})$ will distinguish the two gecs shown in Figure~\ref{fig.qm}, but $\calM(\gec{G})$ does not.

\begin{figure}
\includegraphics[width=8cm]{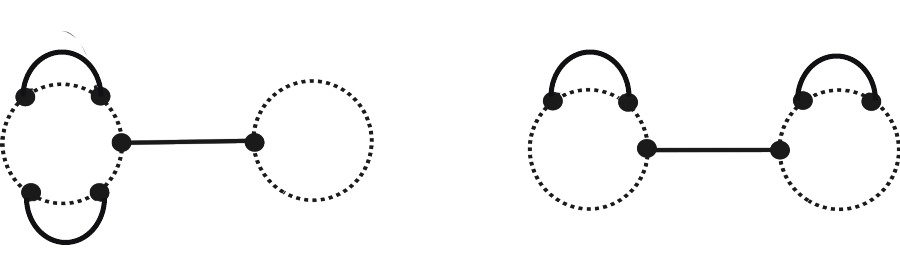}

\caption{Two gecs distinguished by $\calM(\gec{G})$ but not by $\Q(\gec{G})$.}
\label{fig.mq}
\end{figure}

\begin{figure}
\includegraphics[width=8cm]{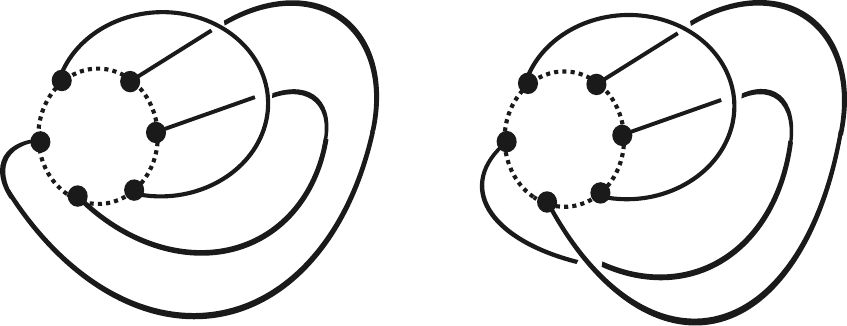}

\caption{Two gecs distinguished by $\Q(\gec{G})$ but not by $\calM(\gec{G})$.}
\label{fig.qm}
\end{figure}

\section{Cog invariants via the Yamada polynomial}\label{ss.yam}
 
In this section we obtain cog invariants via knot theory. The basic idea is to construct a cog invariant by starting with a suitable invariant, here the Yamada polynomial for cogs embedded in 3-space, and then restricting the invariant to values that do not depend upon the particular embedding of the cog in 3-space. We begin with a brief overview of spatial graphs and flat vertex graphs following~\cite{yamada89}.

A \emph{spatial graph} $\spa{G}$ is a graph embedded in $\mathbb{R}^3$. If for each vertex $v$ of $\spa{G}$ there exists a neighbourhood $B_v$ of $v$ and a plane $P_v$ such that $\spa{G} \cap B_v \subset P_v$ then we say that $\spa{G}$ is a \emph{flat vertex graph}.\footnote{Flat vertex graphs are sometimes called rigid vertex graphs in the literature. However, the term rigid vertex graph is also used in the literature to describe what are here called cogs. Ordinarily the potential for confusion is limited as only one of these two mathematical objects is considered. However, as we consider both objects we shall avoid using the term rigid vertex graph.}
Two flat vertex graphs $\spa{G}$ and $\spa{G}'$ are \emph{flatly isotopic} if there exists an isotopy $h_t: \mathbb{R}^3\rightarrow \mathbb{R}^3$, for $t\in[0,1]$, such that $h_0=\mathrm{id}$, $h_1(\spa{G})=\spa{G}'$, and each  $h_t(\spa{G})$ is a flat vertex graph. 

A projection $p:\mathbb{R}^3\rightarrow \mathbb{R}^2$ is a regular projection of a flat vertex graph $\spa{G}$ if each multipoint of $p(\spa{G})$ is a double point where two edges cross transversely. The image $p(\spa{G})$  together with information about the over crossing at all double points of $p(\spa{G})$ is a \emph{diagram} of $\spa{G}$. Examples of flat vertex graph diagrams can be found in Figure~\ref{fig:yamada}. 

It is likely that the reader has some familiarity with Reidemeister's Theorem in knot theory which states that two link diagrams  represent  ambient isotopic links if and only if the two link diagrams are related by a finite sequence of Reidemeiser Moves. An analogous result holds for flat vertex graphs. The result is given in~\cite[Lemma~1]{yamada89} which gives that two flat vertex graphs are flatly isotopic if and only if their diagrams are related by a set of \emph{flat deformation moves}. As we do not need their details here, we shall not write down the flat deformation moves and instead refer the reader to~\cite[Figure~1, Items~(I)--(V)]{yamada89} which defines them. (See also~\cite[Section~III]{Kau89} which details the result in the 4-regular case.) 

\medskip

We now return to cogs. 
A flat vertex graph is a spatial graph where there are planes containing small neighbourhoods of the vertices. We can choose an orientation of each of these planes to determine a cyclic order of the half-edges incident to each vertex. Modulo the two possible orientations of each plane we obtain an undirected cyclic order at each vertex, and hence a cog, which we call the \emph{underlying cog} of the flat vertex graph.
Thus, we may view a flat vertex graph as an embedding in $\mathbb{R}^3$ of its underlying cog.

In terms of diagrams, the underlying cog is obtained from the flat vertex graph diagram by reading the cyclic orders of incident half-edges according to an orientation of the plane to obtain a rotation system, then taking the underlying cog of this rotation system. Observe that the flat deformation move~(V) in~\cite[Figure~1]{yamada89} means that  cyclic orders can be reversed, but the cog's underlying graph is unchanged by the other flat deformation moves. Again we see that a flat vertex graph diagram determines a cog.

An immediate observation is that as flatly isotopic flat vertex graphs give rise to the same underlying cog, cog invariants, such as those defined in previous sections, do define flat vertex graph invariants, albeit rather coarse ones that ignore much of the information about the embedding in $\mathbb{R}^3$. However, here our interest is in the other direction, using known flat vertex graph invariants to obtain cog invariants.

\begin{figure}[t]
    \centering
    \includegraphics[width=.1\linewidth]{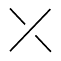}
   \quad \raisebox{5mm}{$\longleftrightarrow$} \quad
    \rotatebox{90}{\includegraphics[width=.1\linewidth]{crossing}}
    \caption{A crossing change.}
    \label{fig:crossch}
\end{figure}

A \emph{crossing change} in a flat vertex graph diagram is the move that swaps the under- and over-crossing arcs at a double point in the diagram, as in Figure~\ref{fig:crossch}. This corresponds to allowing edges of flat vertex graphs in $\mathbb{R}^3$ to pass through each other. It is not hard to see that two flat vertex graphs are related by flat isotopy and passing edges through each other if and only if they have the same underlying cog. 
This means that if we have an invariant $I$ of flat vertex graphs  whose value does not change under crossing changes in their diagrams then $I$ is a cog invariant. We shall now use the Yamada polynomial to obtain two cog invariants via this process. 

\medskip

The \emph{Yamada polynomial}, introduced in~\cite{yamada89}, is a polynomial of flat vertex graphs. By \cite[Proposition~3]{yamada89}), its value $\spa{R}(\spa{G})\in \mathbb{Z}[A^{\pm1}]$ on a flat vertex graph diagram $\spa{G}$ can be computed recursively by the skein relation
\begin{equation}\label{eq:ydef1}
\spa{R}\Big(\raisebox{-.25cm}{\includegraphics[width=.7cm]{crossing}}\Big) = A\, \spa{R}\Big(\raisebox{-.25cm}{\includegraphics[width=.7cm]{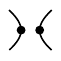}}\Big) + A^{-1}\, \spa{R}\Big(\raisebox{-.25cm}{\includegraphics[width=.7cm]{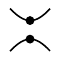}}\Big) + \spa{R}\Big(\raisebox{-.25cm}{\includegraphics[width=.7cm]{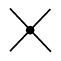}}\Big),
\end{equation}
together with the deletion-contraction relation
\begin{equation}\label{eq:ydef2}
    \spa{R}(\spa{G}) = \spa{R}(\spa{G}\ba e) + \spa{R}(\spa{G} / e), \quad \text{for any non-loop edge $e$,}
\end{equation}
and rules 
\begin{align}
     \spa{R}(\spa{G}_1 \sqcup \spa{G}_2) &= \spa{R}(\spa{G}_1)\cdot \spa{R}(\spa{G}_2),\label{eq:ydef3}\\
     \spa{R}(\spa{G}_1 \vee \spa{G}_2) &= -\spa{R}(\spa{G}_1)\cdot \spa{R}(\spa{G}_2),\label{eq:ydef4}
     \\
     \spa{R}(\bullet)&=-1, \label{eq:ydef5}
    \\
    \quad \spa{R}(\spa{B}_1) &= \sigma,\label{eq:ydef6}
\end{align}
where $\sigma = (A+1+A^{-1})$. 
Here the four diagrams in the skein relation~\eqref{eq:ydef1} are identical except in the regions shown: $\spa{G}\ba e$ is the diagram obtained by removing the edge $e$ from $\spa{G}$; $\spa{G}/ e$ is the diagram obtained by identifying the edge $e$ in $\spa{G}$ to a point (by forming the quotient space $\mathbb{R}^2/\{e\}$), with that point forming a new vertex;
$\spa{G}_1 \sqcup \spa{G}_2$ consists of a  diagram in which there is a simple closed curve separating $\mathbb{R}^2$ into two regions one containing   $\spa{G}_1$ the other $\spa{G}_2$; $\spa{G}_1 \vee \spa{G}_2$ is a diagram in which  there is a simple closed curve intersecting a single vertex of the graph  and separating $\mathbb{R}^2$ into two regions one containing all of  $\spa{G}_1$ the other all $\spa{G}_2$ (except for the vertex on the separating curve); $\bullet$ is the diagram of a single vertex; and $\spa{B}_1$ is the diagram consisting of one vertex, one edge and no crossings.

The polynomial $\spa{R}(\spa{G})$ is not invariant under all of the flat deformation moves and hence does not define a flat isotopy invariant of flat vertex graphs.
However, a normalisation of it does. 
We have that 
\[\overline{\spa{R}}(\spa{G})= (-A)^{-m}\,\spa{R}(\spa{G}), \]
where $m$ is the degree of the lowest order term in $\spa{R}(\spa{G})$ (so  $\overline{\spa{R}}(\spa{G})\in \mathbb{Z}[A]$),
is invariant under all the flat deformation moves and hence defines a flat isotopy invariant of flat vertex graphs (see \cite[Theorem~3]{yamada89}).

\medskip

Now we describe our cog invariants.
We shall use $\overline{\spa{R}}_p(\spa{G})$ (respectively, $\spa{R}_p(\spa{G})$) to denote the value of  $\overline{\spa{R}}(\spa{G})$ (respectively, $\spa{R}(\spa{G})$) evaluated at $A=p$.
Observe from the skein relation~\eqref{eq:ydef1} for $\spa{R}$ that when $A \in \{-1, 1\}$ we have 
$
\spa{R}\Big(\raisebox{-.25cm}{\includegraphics[width=.7cm]{crossing}}\Big) = 
\spa{R}\Big(\raisebox{-.25cm}{\rotatebox{90}{\includegraphics[width=.7cm]{crossing}}}\Big).
   $
Thus both $\overline{\spa{R}}_{-1}(\spa{G}) = \spa{R}_{-1}(\spa{G})$ and $|\overline{\spa{R}}_{1}(\spa{G})| = |\spa{R}_1(\spa{G})|$ are invariant under both flat isotopy and crossing changes and hence give cog invariants. We record this in the following proposition.

\begin{proposition}\label{prop:ya}
Let $\spa{G}$ and $\spa{G}'$ be flat vertex graph diagrams. If $\spa{G}$ and $\spa{G}'$ determine the same underlying cog
then 
\[ \spa{R}_{-1}(\spa{G})= \spa{R}_{-1}(\spa{G}')\quad \text{and} \quad |\spa{R}_{1}(\spa{G})|= |\spa{R}_{1}(\spa{G}')|.\]
That is, $\spa{R}_{-1}$ and $|\spa{R}_{1}|$ define cog invariants.
\end{proposition}

\begin{example}\label{examp:yasp}
Let $\spa{G}_1$ and  $\spa{G}_2$ be the flat vertex graph diagrams in Figure~\ref{fig:yamada}. Then
\[
    \spa{R}(\spa{G}_1) = \spa{R}(\spa{B}_1 \vee \spa{B}_1) = -\sigma^2  =-(A+1+A^{-1})^2,
\]
and, writing  $\spa{T}_4$ to denote the plane graph consisting of two vertices and   four parallel edges,
\begin{align*}
    \spa{R}(\spa{G}_2) &=  A\, \spa{R}(\spa{B}_1 \vee \spa{B}_1) +A^{-1} \, \spa{R}(\spa{B}_1 \vee \spa{B}_1) +  \spa{R}(\spa{T}_4)
    \\&= A (-\sigma^2)+A^{-1} (-\sigma^2)+ ( \sigma^3 -\sigma^2 +\sigma)
    \\&
    =\sigma(-\sigma^2) +\sigma^3+\sigma 
    = \sigma = A+1+A^{-1}.
\end{align*}
This gives
$|\spa{R}_{1}(\spa{G}_1)|=9$ and $|\spa{R}_{1}(\spa{G}_2)|=3$. Thus, as a cog invariant, $|\spa{R}_{1}|$ separates the non-isomorphic cogs with one vertex and two loops.  However
 $\spa{R}_{-1}(\spa{G}_1)=-1=\spa{R}_{-1}(\spa{G}_2)$, so $\spa{R}_{-1}$ does not separate these two cogs.

\end{example}

\begin{figure}[t]
    \centering
    \labellist
\pinlabel $\spa{G}_1$ at 28 6
\pinlabel $\spa{G}_2$ at 96 6
\endlabellist
    \includegraphics[width=.4\linewidth]{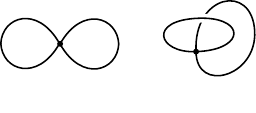}
    \caption{Two flat vertex graph diagrams.}
    \label{fig:yamada}
\end{figure}

Example~\ref{examp:yasp} shows that $|\spa{R}_{1}|$ is a genuine cog invariant as opposed to just a graph invariant. This is since the cogs $\spa{G}_1$ and $\spa{G}_2$ in the example have the same underlying graph but are distinguished by $|\spa{R}_{1}|$. However, the example does not show this for $\spa{R}_{-1}$. It turns out that $\spa{R}_{-1}$ depends only upon the underlying graph of a cog, hence it is a  graph invariant.

\begin{proposition}
Let $\spa{G}$ and $\spa{G}'$ be flat vertex graph diagrams with the same underlying graph, 
then 
\[ \spa{R}_{-1}(\spa{G})= \spa{R}_{-1}(\spa{G}').\]
Thus, $\spa{R}_{-1}$  defines a graph invariant and thus does not distinguish cogs that have the same underlying graph.
\end{proposition}
\begin{proof}
\cite[Proposition~4(3)]{yamada89} with $A=-1$ gives that
\[ \spa{R}_{-1} \left( \raisebox{-3mm}{\includegraphics[height=10mm]{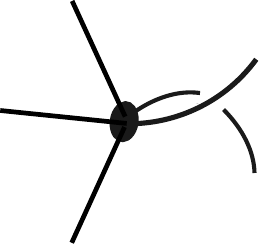}} \right) =  \spa{R}_{-1} \left( \raisebox{-3mm}{\includegraphics[height=10mm]{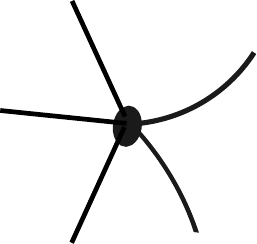}} \right).\]
It follows that we can arbitrarily change the cyclic order at any vertex without changing the value of $\spa{R}_{-1}$.
\end{proof}

As our interest is in cog invariants, and $\spa{R}_{-1}$ does not depend upon  the cyclic orders in a cog, for the rest of this section we shall focus only on $|\spa{R}_{1}|$.

\begin{remark}
As we know that $|\spa{R}_{1}|$ defines a cog invariant, it is natural to ask if taking the absolute value is necessary. That it is necessary can be seen by considering the example given on the last page of~\cite{yamada89} where evaluating the two polynomials at 1 gives different values even though the flat vertex graphs determine the same underlying cog (as all the vertices are of degree three).
\end{remark}

\medskip

We now turn our attention to finding a combinatorial description of  $|\spa{R}_{1}|$.
Given a flat vertex graph diagram $\spa{G}$, first observe we can use Equation~\eqref{eq:ydef1} to eliminate all crossings and  hence (possibly also using Equation~\eqref{eq:ydef3})  express $\spa{R}(\spa{G})$ in terms of values of $\spa{R}$ on plane graphs (i.e., graphs cellularly embedded in $\mathbb{R}^2$). Next observe that if $\spa{G}'$ is a flat vertex graph diagram with no crossings, i.e., it is a graph embedded in the plane, then Equations~\eqref{eq:ydef2}--\eqref{eq:ydef6} define a graph polynomial via deletion-contraction relations, and by the Universality Property of the Tutte polynomial (a statement of which can be found in~\cite[Theorem~2.24]{zbMATH07553843}) it is an evaluation of the Tutte polynomial:
$ \spa{R}(\spa{G}') = (-1)^{k(\spa{G}')} T(\spa{G}'; 0,-\sigma)$, where $k(\spa{G}')$ is the number of components of $\spa{G}'$, and $T(\spa{G}';x,y)$ denotes the Tutte polynomial.
Combining these two observations allows us to use graph states to rewrite $\spa{R}(\spa{G})$ as follows.

\begin{table}[t]
    \centering

\begin{tabular}{|c|c|c|c|}
\hline
 \raisebox{14pt}{$s\in\mathfrak{S}(\spa{G})$}& 
   \labellist
   \small\hair 2pt
\pinlabel $+1$ at 14 4
\endlabellist
   \includegraphics[width=.1\linewidth]{crossing}
 &    \labellist
   \small\hair 2pt
\pinlabel $0$ at 14 4
\endlabellist \includegraphics[width=.1\linewidth]{crossing}&
  \labellist
   \small\hair 2pt
\pinlabel $-1$ at 14 4
\endlabellist
   \includegraphics[width=.1\linewidth]{crossing}
\\
\hline
  \raisebox{14pt}{$\spa{G}[s]$}  & \includegraphics[width=.1\linewidth]{smoothing_v}
 &   \includegraphics[width=.1\linewidth]{vertex}&
   \includegraphics[width=.1\linewidth]{smoothing_h}
\\\hline
\end{tabular}
    \caption{states of a flat vertex graph diagram.}
    \label{table:states}
\end{table}

By a \emph{state} of a flat vertex graph diagram $\spa{G}$ we mean the assignment of the symbols $+1$, $0$, or $-1$ to each of its crossings. We let $\mathfrak{S}(\spa{G})$ denote the set of all states of $\spa{G}$. For a state $s$ of $\spa{G}$ we let $w(s)$ denote the sum of the symbols $+1,0,-1$ over all crossings, and we let $\spa{G}[s]$ denote the graph embedded in the plane obtained by replacing each crossing with the configuration indicated in Table~\ref{table:states}. (So a crossing assigned 0 becomes a degree four vertex, and crossings with a $+$ or $-$ are smoothed as indicated in the figure.) We can then write
\begin{equation}\label{eq:yamtutte}
\spa{R}(\spa{G}) = \sum_{s\in \mathfrak{S}(\spa{G})} A^{w(s)} (-1)^{k(\spa{G}[s])} T(\spa{G}[s]; 0, -\sigma).
\end{equation}
(Note that this expression for $\spa{R}(\spa{G})$ is just that appearing at the end of page~540 of~\cite{yamada89} after using~\cite[Theorem~2]{oxley} to express the Negami polynomial in terms of the Tutte polynomial.)
We shall make use of a known interpretations of the Tutte polynomial at $-\sigma=-3$ (when $A=1$).

\medskip

The cogs invariant arising from $|\spa{R}_{1}|$ can be described without reference to spatial graphs by considering cog drawings. 
A cog drawing is just a graph drawing, in the usual sense, in an oriented plane, but we use the different name to emphasise that the context is different.
A \emph{cog drawing} consists of a set of points, called \emph{vertices}, in an oriented plane; and a set of paths in this plane, called \emph{edges}, between the vertices. The edges do not intersect vertices except at their end points and all other multiple points consist of two paths crossing transversally. Such multiple points are called \emph{crossings}. 

The vertices and edges of a cog drawing $\dr{\cog{G}}$ give rise to a graph in the usual and obvious way. In addition, the orientation of the plane specifies a cyclic order of the incident half-edges at each vertex.
Thus, we obtain a rotation system, and we then take its underlying cog $\cog{G}$ (so we are considering a cog as a reversal-equivalence class of rotation systems).
We say that $\dr{\cog{G}}$ is a \emph{drawing of $\cog{G}$}. 

A \emph{state} $\dr{S}$ of a cog drawing $\dr{\cog{G}}$ is a plane graph arising by replacing each crossing in $\dr{\cog{G}}$ with one of the configurations indicated in the bottom row of Table~\ref{table:states}. We use $\mathfrak{S}( \dr{\cog{G}})$ to denote the set of all states of $\dr{\cog{G}}$. 

\begin{definition}\label{def:yamcogv2}
Let $\cog{G}$ be a cog and $\dr{\cog{G}}$ be any drawing of $\cog{G}$. 
\[ Y(\cog{G}) =  \left|\sum_{\dr{S}\in \mathfrak{S}(\dr{\cog{G}})}  (-1)^{k(\dr{S})} \; T(\dr{S}; 0, -3) \right|, \]
where
 $k(\dr{S})$ its number of components of $\dr{S}$.  
\end{definition}

Proposition~\ref{prop:ya} together with Equation~\eqref{eq:yamtutte} immediately give that $Y$ ($=|\spa{R}_{1}|$) is a well-defined cog invariant.

\begin{proposition}
The value of 
$Y(\cog{G})$ is independent of the choice of drawing of the cog $\cog{G}$. Additionally, if $\cog{G}$ and $\cog{G}'$ are isomorphic cogs then $Y(\cog{G})=Y(\cog{G}')$. 
\end{proposition}

Note that if a cog $\cog{G}$ has  a plane drawing then $Y(\cog{G})$ reduces to just $(-1)^{k(G)}$ times the Tutte polynomial $T(G;0,-3)$ of the underlying graph $G$.

If $H$ is a graph, then the number of \emph{nowhere-zero $\mathbb{Z}_k$-flows}, denoted $F(H;k)$ is defined to be the number of ways to assign non-zero elements of $\mathbb{Z}_k$ to each edge of an arbitrary orientation of $H$ such that the net flow into each vertex equals the net flow out of the vertex. By~\cite{zbMATH03087501} (or see \cite{zbMATH07680496} for background on flows and the Tutte polynomial), we have $F(H;k) = (-1)^{n(H)}T(H;0,1-k)$. Thus, the cog invariant $Y(\cog{G})$ can be expressed in terms of counts of nowhere-zero $\mathbb{Z}_4$-flows:
\begin{equation*}\label{eq:yamtutteb}
Y(\cog{G})= \Big|\sum_{\dr{S}\in \mathfrak{S}(\dr{\cog{G}})} (-1)^{e(\dr{S})+v(\dr{S})}\, F(\dr{S},4)\Big|,
\end{equation*}
where $\dr{\cog{G}}$ is any cog drawing of $\cog{G}$.

\section{Further study}\label{sec:further}

As noted in the introduction, we aimed to begin a study of the topological properties of cogs as a class of combinatorial objects situated  between abstract and cellularly embedded graphs.  That cogs lie between these classes suggests numerous further directions for study, and we mention a few possibilities here.

\begin{enumerate}


\item \label{fs.1} Since genus questions are central to topological graph theory, it is natural to consider the genus range of a cog. This question also exposes another way that cogs sit between abstract and embedded graphs. For this item we assume all graphs are connected.

The genus range of an abstract graph, that is, the set of genera of the surfaces in which the graph cellularly embeds, is known to be a  sequence of consecutive integers. 
We may consider different notions of genus in this definition. The \emph{orientable genus range} considers only embeddings in orientable surfaces; The \emph{nonorientable genus range} considers only embeddings in nonorientable surfaces; the \emph{Euler genus range} considers the Euler genus over embeddings in both orientable and nonorientable surfaces. (The \emph{Euler genus} equals the genus for nonorientable surfaces and equals twice the genus of an orientable surface. By Euler's formula it equals $2k-v+e-f$.)  

The genus range of a cog can be defined analogously: it is the set of genera of cellular graph embeddings with given underlying cog.  ``Genus'' here can mean ``orientable genus'', ``nonorientable genus'', or ``Euler genus''.
Each genus range of a cog is a subset of the corresponding genus range of its underlying graph.

While the orientable genus range of an abstract graph is always an interval, this is not the case for cogs. 
The $n$-dipole graphs, which consist of two vertices with $n$ parallel edges between them, give rise to a family of cogs with arbitrarily large gaps in their orientable genus ranges. 
If we take the cog of the usual planar embedding of an $n$-dipole, with $n$ faces of length $2$, then the other orientable embedding of this has one face (and hence genus $(n-1)/2$) for odd $n$ and two faces (and hence genus $(n-2)/2$) for even $n$.

In contrast the Euler genus range of a cog is always an interval.
To see this begin with any minimum Euler genus embedding of the cog.  We may then take the partial Petrie dual with respect to any edge that meets two faces.  Since taking partial Petrie duals does not change the cyclic orders at the vertices, the result of this operation is again an embedding of the original cog, but now with one less face.  Repeating this process reduces the number of faces exactly one at a time, until a one-face embedding of the cog results.  This gives an interval of genera from the minimum to the maximum Euler genus embeddings of the cog.
 Moreover, combining two faces by taking the partial Petrie dual of an edge contained in both always results in a nonorientable embedding \cite[Lemma 6(a)]{Sta78}.  It follows that if the range of the Euler genus is the interval $[g, h]$, then the range of the nonorientable genus is also an interval, either $[g, h]$ or $[g+1, h]$ (which is empty if $g=h$).

A natural question is for which cogs the orientable genus range is an interval.  An immediate example is given by $3$-regular graphs, as there is only one undirected cyclic ordering of the half-edges at each vertex, so the orientable genus range is just the orientable genus range of the underlying graph.
It is shown in~\cite[Lemma 4]{B+15} (where cogs are called ``rigid vertex graphs'') that 4-regular cogs also have orientable genus ranges that are intervals.
In fact, we can provide a short proof of the following more general result.

\begin{proposition} The orientable genus range of a cog of maximum degree at most $4$ is an interval.
\end{proposition}

\begin{proof}
The orientable genus range is the same as the oriented genus range.
The oriented embeddings for the given cog form a reversal-equivalence class.
Consider one such oriented embedding.
Reversal at a vertex $v$ may change the number of faces incident with $v$, although the number of faces not incident with $v$ does not change.  But the number of faces incident with $v$ is a number in the range $[1,4]$, and it can only change by an even number since the Euler characteristic must be even.  So the only possible changes are $0$ and $\pm 2$, which means the orientable genus changes by $0$ or $\pm 1$. Therefore, by starting with an oriented embedding of minimum genus for a cog, and performing a sequence of vertex reversals to reach an oriented embedding of maximum genus for that cog, we must go through an interval of values.  Thus, the oriented, or orientable, genus range is an interval.
\end{proof}

Classical questions in topological graph theory have cog analogues. For example finding a characterisation for upper-embeddable cogs, i.e.~those with orientable embeddings having only one face, analogous to that for abstract graphs (see \cite{MR532590} and \cite{MR286713}). Or, more generally, determining the maximum orientable genus of cogs.

\item  \label{fs.2} In Section~\ref{sec:trans2} we considered cog invariants arising from the topological transition polynomial. Cog invariants also arise from Jaeger's transition polynomial~\cite{MR1096990} and the 
 generalised transition polynomial of  \cite{MR1980048}, both of which also use vertex transitions. Details can be found in \cite[Section 13.3.5]{zbMATH07680507}. In brief, if $G$ is an even graph, possibly with free loops, then  
a \emph{vertex state} at a vertex $v$ of $G$ is a partition, into pairs, of the edges incident with $v$. A \emph{graph state} is a choice of vertex state at each vertex. 
A graph state $s$ defines a cover of $G$ by circuits (each circuit given by following the paired edges). We let $c(s)$ denote the number of such circuits resulting \mbox{from $s$.}  

If there is a way to distinguish among these vertex states (for example, by labelling half-edges, by using an underlying orientation, by a checkerboard colouring, etc.), then we may use such distinctions to assign a weight $\omega(v,s)$ to the vertex state at $v$ in a graph state $s$, and then a weight $\omega(s) = \prod_{v \in V(F)} {\omega(v,s)}$ to $s$. This results in a weight function $W$ that assigns a weight to each graph state.
The \emph{transition polynomial} is then
\begin{equation}\label{EMMmedial:sdhi}
q(G; W,t)= \sum_{s} {\omega( s ) \,t^{c(s)} },
\end{equation} 
where the sum  is over all graph states $s$ of $G$.  

Key to the definition of the transition polynomial is that there needs to be a mechanism to distinguish among these vertex states.
The observation is then that if $\cog{G}$ is an even cog, then the undirected cyclic orders can be used to distinguish between the vertex states, and the transition polynomial will give rise to a cog invariant as long as the values of $\omega(v,s)$ depend only upon the undirected cyclic orders. 
(For example, if the undirected cyclic order at a degree four vertex $v$ is $\ouco e_1,e_2,e_3,e_4\cuco$ then the vertex states given by  $(e_1,e_2)$,  $(e_3,e_4)$ and by $(e_1,e_4)$,  $(e_3,e_2)$ must have the same weight, which can be different from that given by $(e_1,e_3)$,  $(e_2,e_4)$.)  Thus, the generalised transition polynomial of \cite{MR1980048} can also immediately give  cog invariants. What are the properties of these invariants and what can we deduce about cogs from them?

Similarly the Holant graph invariants introduced in \cite{ MR2386281} and developed in \cite{MR2988772}  (see also \cite{zbMATH07680491} for a survey with graph polynomial connections) use functions dependent on properties of the half-edges incident to a vertex, such as cyclic order.
These invariants might then also be readily adapted to cog invariants.

\item  \label{fs.3}
An invariant of the type discussed in (\ref{fs.2}) above arises in the context of assembly graphs, which have been used to model DNA recombination~\cite{angeleska09}. 
Following~\cite{B+13}, an \emph{assembly graph} is a rotation system in which each vertex is of degree one or four. (Note that~\cite{angeleska09} defines assembly graphs as cogs rather than rotation systems, but  this difference in the definitions proves to be immaterial here.)
An assembly graph is called \emph{simple} if it contains   a \textit{transverse Eulerian trail}, which an Eulerian trail that goes straight across every degree four vertex, as determined by the cyclic order of its half-edges. Note it has exactly two degree one vertices. It is a \emph{simple oriented assembly graph} if this transverse Eulerian trail is directed.  

In \cite{B+13} the authors give a polynomial invariant $S_{\ass{G}}(p,t) \in \mathbb{Z}[p,t]$ of a simple oriented assembly graph $\ass{G}$ called the \textit{assembly polynomial}.
The assembly polynomial is an instance of the transition polynomial of Equation~\eqref{EMMmedial:sdhi}. If $\hat{\ass{G}}$ is the rotation system obtained from a simple oriented assembly graph $\ass{G}$ by identifying the two degree one vertices, then 
\[S_{\ass{G}}(p,t) = t^{-1} q(\hat{\ass{G}}; W,t)\] where  $\omega(v, s )$ is given as follows.
 If $e_1$, $e_2$, $e_3$, $e_4$ are the half-edges incident to a degree four vertex $v$ and the transverse directed transverse Eulerian trail meets the half-edges in the order $e_1,e_2,e_3,e_4$ (thus the cyclic order at $v$ is either $( e_1,e_3,e_2,e_4)$ or $( e_1,e_4,e_2,e_3)$), then set $\omega(v, s )=p$ when $s$ is the pairing $(e_1,e_4),(e_2,e_3)$, set $\omega(v, s)=1$ when $s$ is $(e_1, e_3),(e_2, e_4)$,  set  $\omega(v, s )=0$ when $s$ is $(e_1,e_2),(e_3,e_4)$. If $v$ is the degree 2 vertex then set $\omega(v, s)=1$. (Note that if we use  extension of the transition polynomial to non-even graphs, such as in \cite{MR2693850}, then we can consider the transition polynomial of $\ass{G}$ rather than $\hat{\ass{G}}$.) 
In determining the value of $\omega(v, s )$ we do need to make use of the transverse Eulerian trail, however when $p=1$ the values of $\omega(v, s )$ can be defined in terms of the cyclic order of half-edges at $v$ alone.  
Indeed the value of $(v, s )$ can be defined in terms of the undirected cyclic order  $\ouco e_1,e_3,e_2,e_4\cuco$ alone. It follows that  $S_{\ass{G}}(1,t)$ gives an invariant for a class of cogs.
(Note that $S_{\ass{G}}(1,t)$ can be defined for cogs that have only degree one or four vertices. When $p=1$, the existence of a transverse Euler trail is no longer needed.)  
 What are the properties of this invariant and what does it tell you about the original biological problem? More generally, what is the role of the theory of cogs in DNA applications?

\item  \label{fs.4} 
Multimatroids (see~\cite{MM1zbMATH01116184}) constructed from delta-matroids are another possible source of cog invariants.

Every ribbon graph $\bG$ has an associated \emph{delta-matroid} $D(\bG)$, which consists of a family of edge sets, known as \emph{feasible sets}, satisfying a certain exchange property \cite{MR904585,zbMATH04185622,CMNR1}.
Two ribbon graphs that are partial duals have delta-matroids that are \emph{twists} of each other.
Every delta-matroid has an associated \emph{$2$-matroid} (or \emph{symmetric matroid}), and delta-matroids that are twists have the same $2$-matroid \cite{MR904585}.  Thus, each equivalence class of ribbon graphs under partial duality has an associated $2$-matroid.

By Proposition \ref{prop:ppcog2}, each cog corresponds to an equivalence class of ribbon graphs under partial Petrie duality.
From the theory of twisted duals (see~\cite{MR2869185}) we have
$(\bG^{\tau(A)})\wil = (\bG\wil)^{\delta(A)}$ for a given ribbon graph $\bG$, where $\bG\wil =  \bG^{\delta\tau\delta(E)}$ is called the \emph{Wilson dual} of $\bG$.  Hence, the Wilson dual maps the partial Petrie duality equivalence class of $\bG$ to the partial duality class of $\bG\wil$, giving a 1-1 correspondence between partial Petrie duality classes of ribbon graphs and partial duality classes of ribbon graphs.  We can therefore map a cog $\cog{G}$ to the corresponding partial Petrie duality class, then via the Wilson dual to a partial duality class, and hence to a $2$-matroid.  An invariant of the $2$-matroid is an invariant of the cog.  Different cogs may have the same $2$-matroid, but the situations in which this occurs are restricted, similar to the situations in which two graphs have the same matroid.

An obvious 2-matroid invariant to consider is the interlace polynomial or transition polynomial~\cite{MR3191496}. However, due to the compatibility between theses polynomials and the transition polynomial of a ribbon graph as described in~\cite{CMNR2} this approach essentially recreates the construction of the invariant in Equation~\eqref{eq:qgeq1}. 
Although we were unable to find genuinely new examples of cog invariants arising through multimatroids in this way, we believe some should exist.

The $2$-matroid $S$ of a cog $\cog{G}$ described above can be derived directly from a ribbon graph with underlying cog $\cog{G}$, without going through Wilson duals.
However, $S$ does not seem a completely natural object, in the following sense.
Cogs sit between graphs and ribbon graphs, which have very fruitful extensions to matroids and delta-matroids respectively, and the delta-matroid of a ribbon graph $\bG$ contains the matroid of its underlying graph $G$ in an obvious way.
However, if $\cog{G}$ is the cog between $\bG$ and $G$, then its $2$-matroid $S$ is not contained in an obvious way in the delta-matroid of $\bG$, and does not contain an obvious copy of the matroid of $G$.
A tantalising question, therefore, is whether there is a class of combinatorial objects, between matroids and delta-matroids, that extends cogs in a more natural way. 

\item\label{fs.sc}
When we form a cog from a signed rotation system, or from a graph embedding (considered as a flip-equivalence class of signed rotation systems) we discard all edge signs.  But it may make sense to keep some sign information.  We could consider \emph{signed cogs}, namely cogs with an edge signature, or cogs with an underlying signed graph, but these do not seem to have a natural interpretation in terms of graph embeddings.  However, signed graphs are often considered up to switching-equivalence, and a switching-equivalence class of signed graphs can be considered as a graph with a group homomorphism from its cycle space (as an additive group) to the multiplicative group $\{1, -1\}$: we will call this a \emph{cycle-signed graph}.  In the same way we can consider switching-equivalence classes of signed cogs as \emph{cycle-signed cogs}.  Cycle-signed cogs can be shown to be in 1-1 correspondence with equivalence classes of general (orientable or nonorientable) graph embeddings under vertex reversals (which are well-defined operations on general embeddings because they commute with vertex flips).  Therefore, construction of invariants for cycle-signed cogs that are not just cog invariants may be of some interest.

\end{enumerate}

\section*{Acknowledgments}

This work arose from discussions between M.E, J.E.-M., I.M and W.M. at the \emph{Workshop on Uniqueness and Discernment in Graph Polynomials} held at the MATRIX Institute for the Mathematical Sciences in October 2023. M.E., J.E.-M. and I.M. were partially supported by MATRIX--Simons Travel Grants to participate in this workshop.  We are very grateful for the support and productive research environment provided by MATRIX.

M.E. is grateful to the Simons Foundation for support under awards 429625 and MPS-TSM-00002760, and thanks Kevin Grace for useful discussions.

\section*{Statements on open access:}

For the purpose of open access, the authors have applied a Creative Commons
Attribution (CC BY) license to any Author Accepted Manuscript version arising.

There are no conflicts of interest.

\bibliographystyle{abbrv}

\end{document}